\documentclass[12pt]{amsart}
\usepackage{txfonts}      %{article} was 12pt latex e
\usepackage{amssymb}
\usepackage{eucal}
\usepackage{amsmath}
\allowdisplaybreaks[4]
\usepackage{amscd}
\usepackage{xcolor}
\usepackage{multicol}
\usepackage[all]{xy}           %xypic macro for latex2.09
\usepackage{graphicx}
\usepackage{color}
\usepackage{colordvi}
\usepackage{xspace}
\usepackage{tikz}
\usepackage{makecell}
\usepackage{appendix}
\usepackage{amsthm}
\usepackage[misc]{ifsym}
\usepackage{mathrsfs}

\usepackage{ifpdf}
\ifpdf
\usepackage[colorlinks,final,backref=page,hyperindex]{hyperref}
\else
\usepackage[colorlinks,final,backref=page,hyperindex,hypertex]{hyperref}
\fi

\usepackage[active]{srcltx} %SRC Specials for DVI Searching

\begin{document}
\newtheorem{theorem}{Theorem}[section]
\newtheorem{lemma}[theorem]{Lemma}
\newtheorem{corollary}[theorem]{Corollary}
\newtheorem{proposition}[theorem]{Proposition}
\theoremstyle{definition}
\newtheorem{definition}[theorem]{Definition}
\newtheorem{example}[theorem]{Example}
\newtheorem{remark}[theorem]{Remark}
\newtheorem{pdef}[theorem]{Proposition-Definition}
\newtheorem{condition}[theorem]{Condition}
\renewcommand{\labelenumi}{{\rm(\alph{enumi})}}
\renewcommand{\theenumi}{\alph{enumi}}
\baselineskip=14pt

\newcommand {\emptycomment}[1]{} %to remove paragraphs

\newcommand{\nc}{\newcommand}
\newcommand{\delete}[1]{}

\nc{\todo}[1]{\tred{To do:} #1}

\nc{\tred}[1]{\textcolor{red}{#1}}
\nc{\tblue}[1]{\textcolor{blue}{#1}}
\nc{\tgreen}[1]{\textcolor{green}{#1}}
\nc{\tpurple}[1]{\textcolor{purple}{#1}}
\nc{\tgray}[1]{\textcolor{gray}{#1}}
\nc{\torg}[1]{\textcolor{orange}{#1}}
\nc{\tmag}[1]{\textcolor{magenta}}
\nc{\btred}[1]{\textcolor{red}{\bf #1}}
\nc{\btblue}[1]{\textcolor{blue}{\bf #1}}
\nc{\btgreen}[1]{\textcolor{green}{\bf #1}}
\nc{\btpurple}[1]{\textcolor{purple}{\bf #1}}

%\delete{
	\nc{\mlabel}[1]{\label{#1}}  % Use this to suppress names
	\nc{\mcite}[1]{\cite{#1}}  % Use this to suppress names
	\nc{\mref}[1]{\ref{#1}}  % Use this to suppress names
	\nc{\meqref}[1]{\eqref{#1}}  % Use this to suppress names
	\nc{\mbibitem}[1]{\bibitem{#1}} % Use this to show number
%}

\delete{
	\nc{\mlabel}[1]{\label{#1}  % Use the next two lines to show names
		{ {\small\tgreen{\tt{{\ }(#1)}}}}}
	\nc{\mcite}[1]{\cite{#1}{\small{\tt{{\ }(#1)}}}}  % Use this lines to show names
	\nc{\mref}[1]{\ref{#1}{\small{\tred{\tt{{\ }(#1)}}}}}  % Use this lines to show names
	\nc{\meqref}[1]{\eqref{#1}{{\tt{{\ }(#1)}}}}  % Use this lines to show names
	\nc{\mbibitem}[1]{\bibitem[\bf #1]{#1}} % Use this to show name
}

\nc{\cm}[1]{\textcolor{red}{Chengming:#1}}
\nc{\yy}[1]{\textcolor{blue}{Yanyong: #1}}
\nc{\zy}[1]{\textcolor{yellow}{Zhongyin: #1}}
%\nc{\lit}[2]{\textcolor{blue}{#1}{ \textcolor{purple}{(#2)}}}
%\nc{\lit}[2]{\textcolor{blue}{#1}{}} %use this line instead of the previous one to show only the new changes
\nc{\li}[1]{\textcolor{purple}{#1}}
\nc{\lir}[1]{\textcolor{purple}{Li:#1}}

%%%%%%%% new symbols

\nc{\tforall}{\ \ \text{for all }}
\nc{\hatot}{\,\widehat{\otimes} \,}
\nc{\complete}{completed\xspace}
\nc{\wdhat}[1]{\widehat{#1}}

\nc{\ts}{\mathfrak{p}}
\nc{\mts}{c_{(i)}\ot d_{(j)}}

\nc{\NA}{{\bf NA}}
\nc{\LA}{{\bf Lie}}
\nc{\CLA}{{\bf CLA}}

\nc{\cybe}{CYBE\xspace}
\nc{\nybe}{NYBE\xspace}
\nc{\ccybe}{CCYBE\xspace}

\nc{\ndend}{pre-Novikov\xspace}
\nc{\calb}{\mathcal{B}}
\nc{\rk}{\mathrm{r}}
\newcommand{\g}{\mathfrak g}
\newcommand{\h}{\mathfrak h}
\newcommand{\pf}{\noindent{$Proof$.}\ }
\newcommand{\frkg}{\mathfrak g}
\newcommand{\frkh}{\mathfrak h}
\newcommand{\Id}{\rm{Id}}
\newcommand{\gl}{\mathfrak {gl}}
\newcommand{\ad}{\mathrm{ad}}
\newcommand{\add}{\frka\frkd}
\newcommand{\frka}{\mathfrak a}
\newcommand{\frkb}{\mathfrak b}
\newcommand{\frkc}{\mathfrak c}
\newcommand{\frkd}{\mathfrak d}
\newcommand {\comment}[1]{{\marginpar{*}\scriptsize\textbf{Comments:} #1}}

\newcommand{\da}{\Delta_{\alpha}}
\newcommand{\db}{\Delta_{\beta}}
%%%%%%%%%%%%%%%%%%%%%%% old symbols

\nc{\disp}[1]{\displaystyle{#1}}
\nc{\bin}[2]{ (_{\stackrel{\scs{#1}}{\scs{#2}}})}  %binomial coeff
\nc{\binc}[2]{ \left (\!\! \begin{array}{c} \scs{#1}\\
    \scs{#2} \end{array}\!\! \right )}  %binomial coeff
\nc{\bincc}[2]{  \left ( {\scs{#1} \atop
    \vspace{-.5cm}\scs{#2}} \right )}  %binomial coeff
\nc{\ot}{\otimes}
\nc{\sot}{{\scriptstyle{\ot}}}
\nc{\otm}{\overline{\ot}}
\nc{\ola}[1]{\stackrel{#1}{\la}}%${\Bbb Z}$

\nc{\scs}[1]{\scriptstyle{#1}} \nc{\mrm}[1]{{\rm #1}}

\nc{\dirlim}{\displaystyle{\lim_{\longrightarrow}}\,}
\nc{\invlim}{\displaystyle{\lim_{\longleftarrow}}\,}

\nc{\bfk}{{\bf k}} \nc{\bfone}{{\bf 1}}
\nc{\rpr}{\circ}
%\nc{\apr}{\cdot}
\nc{\dpr}{{\tiny\diamond}}
\nc{\rprpm}{{\rpr}}

%%%%%%%%%%%%%%%%%%%%% roman fonts, in alphabetic order
\nc{\mmbox}[1]{\mbox{\ #1\ }} \nc{\ann}{\mrm{ann}}
\nc{\Aut}{\mrm{Aut}} \nc{\can}{\mrm{can}}
\nc{\twoalg}{{two-sided algebra}\xspace}
\nc{\colim}{\mrm{colim}}
\nc{\Cont}{\mrm{Cont}} \nc{\rchar}{\mrm{char}}
\nc{\cok}{\mrm{coker}} \nc{\dtf}{{R-{\rm tf}}} \nc{\dtor}{{R-{\rm
tor}}}
\renewcommand{\det}{\mrm{det}}
\nc{\depth}{{\mrm d}}
\nc{\End}{\mrm{End}} \nc{\Ext}{\mrm{Ext}}
\nc{\Fil}{\mrm{Fil}} \nc{\Frob}{\mrm{Frob}} \nc{\Gal}{\mrm{Gal}}
\nc{\GL}{\mrm{GL}} \nc{\Hom}{\mrm{Hom}} \nc{\hsr}{\mrm{H}}
\nc{\hpol}{\mrm{HP}}  \nc{\id}{\mrm{id}} \nc{\im}{\mrm{im}}

\nc{\incl}{\mrm{incl}} \nc{\length}{\mrm{length}}
\nc{\LR}{\mrm{LR}} \nc{\mchar}{\rm char} \nc{\NC}{\mrm{NC}}
\nc{\mpart}{\mrm{part}} \nc{\pl}{\mrm{PL}}
\nc{\ql}{{\QQ_\ell}} \nc{\qp}{{\QQ_p}}
\nc{\rank}{\mrm{rank}} \nc{\rba}{\rm{RBA }} \nc{\rbas}{\rm{RBAs }}
\nc{\rbpl}{\mrm{RBPL}}
\nc{\rbw}{\rm{RBW }} \nc{\rbws}{\rm{RBWs }} \nc{\rcot}{\mrm{cot}}
\nc{\rest}{\rm{controlled}\xspace}
\nc{\rdef}{\mrm{def}} \nc{\rdiv}{{\rm div}} \nc{\rtf}{{\rm tf}}
\nc{\rtor}{{\rm tor}} \nc{\res}{\mrm{res}} \nc{\SL}{\mrm{SL}}
\nc{\Spec}{\mrm{Spec}} \nc{\tor}{\mrm{tor}} \nc{\Tr}{\mrm{Tr}}
\nc{\mtr}{\mrm{sk}}

%%%%%%%%%%%%%%%%%% bold face
\nc{\ab}{\mathbf{Ab}} \nc{\Alg}{\mathbf{Alg}}

%%%%%%%%%%%%%%%%%%%Bbb fonts
\nc{\BA}{{\mathbb A}} \nc{\CC}{{\mathbb C}} \nc{\DD}{{\mathbb D}}
\nc{\EE}{{\mathbb E}} \nc{\FF}{{\mathbb F}} \nc{\GG}{{\mathbb G}}
\nc{\HH}{{\mathbb H}} \nc{\LL}{{\mathbb L}} \nc{\NN}{{\mathbb N}}
\nc{\QQ}{{\mathbb Q}} \nc{\RR}{{\mathbb R}} \nc{\BS}{{\mathbb{S}}} \nc{\TT}{{\mathbb T}}
\nc{\VV}{{\mathbb V}} \nc{\ZZ}{{\mathbb Z}}

%%%%%%%%%%%%%%%%%%% cal fonts

\nc{\calao}{{\mathcal A}} \nc{\cala}{{\mathcal A}}
\nc{\calc}{{\mathcal C}} \nc{\cald}{{\mathcal D}}
\nc{\cale}{{\mathcal E}} \nc{\calf}{{\mathcal F}}
\nc{\calfr}{{{\mathcal F}^{\,r}}} \nc{\calfo}{{\mathcal F}^0}
\nc{\calfro}{{\mathcal F}^{\,r,0}} \nc{\oF}{\overline{F}}
\nc{\calg}{{\mathcal G}} \nc{\calh}{{\mathcal H}}
\nc{\cali}{{\mathcal I}} \nc{\calj}{{\mathcal J}}
\nc{\call}{{\mathcal L}} \nc{\calm}{{\mathcal M}}
\nc{\caln}{{\mathcal N}} \nc{\calo}{{\mathcal O}}
\nc{\calp}{{\mathcal P}} \nc{\calq}{{\mathcal Q}} \nc{\calr}{{\mathcal R}}
\nc{\calt}{{\mathcal T}} \nc{\caltr}{{\mathcal T}^{\,r}}
\nc{\calu}{{\mathcal U}} \nc{\calv}{{\mathcal V}}
\nc{\calw}{{\mathcal W}} \nc{\calx}{{\mathcal X}}
\nc{\CA}{\mathcal{A}}

%%%%%%%%%%%%%%%%%%  frak fonts
\nc{\fraka}{{\mathfrak a}} \nc{\frakB}{{\mathfrak B}}
\nc{\frakb}{{\mathfrak b}} \nc{\frakd}{{\mathfrak d}}
\nc{\oD}{\overline{D}}
\nc{\frakF}{{\mathfrak F}} \nc{\frakg}{{\mathfrak g}}
\nc{\frakm}{{\mathfrak m}} \nc{\frakM}{{\mathfrak M}}
\nc{\frakMo}{{\mathfrak M}^0} \nc{\frakp}{{\mathfrak p}}
\nc{\frakS}{{\mathfrak S}} \nc{\frakSo}{{\mathfrak S}^0}
\nc{\fraks}{{\mathfrak s}} \nc{\os}{\overline{\fraks}}
\nc{\frakT}{{\mathfrak T}}
\nc{\oT}{\overline{T}}
%\nc{\frakx}{{\mathfrak x}}
\nc{\frakX}{{\mathfrak X}} \nc{\frakXo}{{\mathfrak X}^0}
\nc{\frakx}{{\mathbf x}}
%\nc{\frakTxo}{{\frakTx}^0}
\nc{\frakTx}{\frakT}      %All rooted trees, correspond to \ncsha(X)
\nc{\frakTa}{\frakT^a}        % rooted trees for \ncsha(A)
\nc{\frakTxo}{\frakTx^0}   % rooted trees for \ncshao(X)
\nc{\caltao}{\calt^{a,0}}   % rooted trees for \ncshao(A)
\nc{\ox}{\overline{\frakx}} \nc{\fraky}{{\mathfrak y}}
\nc{\frakz}{{\mathfrak z}} \nc{\oX}{\overline{X}}

\font\cyr=wncyr10

%%%%%%%%%%%%%%%%%%%%%%%%%%%%%%%%%%%%%%%%%%%%%%%%%%%%%%%%%%%%%%%%%%

%\begin{document}
\title{Infinite-dimensional pre-Lie bialgebras via affinization of pre-Novikov bialgebras}

\author{Yue Li}
\address{School of Mathematics, Hangzhou Normal University,
Hangzhou, 311121, China}
\email{liyuee@stu.hznu.edu.cn}

\author{Yanyong Hong}
\address{School of Mathematics, Hangzhou Normal University,
Hangzhou, 311121, China}
\email{yyhong@hznu.edu.cn}

\subjclass[2010]{17A30, 17A60, 17B38, 17D25}
\keywords{pre-Novikov algebra, pre-Lie algebra, right Novikov dialgebra, affinization, pre-Novikov bialgebra, Yang-Baxter equation}
\footnote{
Corresponding author: Y. Hong (yyhong@hznu.edu.cn).
}
\begin{abstract}
In this paper, we show that there is a pre-Lie algebra structure on the tensor product of a pre-Novikov algebra and a right Novikov dialgebra and the tensor product of a pre-Novikov algebra and a special right Novikov algebra on the vector space of Laurent polynomials being a pre-Lie algebra characterizes the pre-Novikov algebra. The latter is called the affinization of a pre-Novikov algebra. Moreover, we extend this construction of pre-Lie algebras and the affinization of pre-Novikov algebras to the context of bialgebras. We show that there is a completed pre-Lie bialgebra structure on the tensor product of a pre-Novikov
bialgebra and a quadratic $\ZZ$-graded right Novikov algebra. Moreover, a pre-Novikov bialgebra can be characterized by the fact that its affinization by a quadratic $\ZZ$-graded right Novikov algebra on the vector space of Laurent polynomials gives an infinite-dimensional completed pre-Lie bialgebras. Note that  the reason why we choose a quadratic right Novikov algebra instead of a right Novikov dialgebra with a special bilinear form is also given.
Furthermore, we construct symmetric completed solutions of the $S$-equation in the induced pre-Lie algebra by symmetric solutions of the pre-Novikov Yang-Baxter equation in a pre-Novikov algebra.
\end{abstract}

\maketitle
\section{Introduction}
Pre-Lie algebras (or left-symmetric algebras) are a class of Lie-admissible algebras whose commutators are Lie algebras, which
play important roles in many fields in mathematics and mathematical
physics such as convex homogenous cones \cite{V}, affine manifolds and affine structures on Lie groups \cite{Ko,Mat}, deformation of associative algebras \cite{Ger} and so on.
Novikov algebras are pre-Lie algebras whose right multiplications are commutative. Novikov algebras appeared in the study of Hamiltonian
operators in the formal variational calculus \cite{GD1, GD2} and Poisson brackets of hydrodynamic type \cite{BN}. Moreover, by \cite{X1}, Novikov algebras correspond to a class of Lie conformal algebras which describe the singular part of operator product expansion of chiral fields in conformal field
theory \cite{K1}.

The definition of pre-Novikov algebras was introduced in \cite{HBG} in the study of Novikov bialgebras. It was shown in \cite{HBG} that a pre-Novikov algebra naturally  provides a skew-symmetric solution of the Novikov Yang-Baxter equation in the associated Novikov algebra and thus a Novikov bialgebra.
%In addition, pre-Novikov algebras are in correspondence with a class of left-symmetric conformal algebras\cite{XH}.
Moreover, there are close relationships between pre-Novikov algebras and other algebra structures. For example, pre-Novikov algebras are a class  of $L$-dendriform algebras \cite{BLN}, correspond to a class of left-symmetric conformal algebras \cite{HL, XH}, and are closely related to Zinbiel algebras  with derivations (see \cite{HBG, KMS} ).

It was shown in \cite{BN} that there is a natural Lie algebra structure on the tensor product of a Novikov algebra and the vector space ${\bf k}[t,t^{-1}]$, which is called the affinization of a Novikov algebra, and this affinization also characterizes the Novikov algebra. It can be stated in detail as follows.
\begin{theorem}\label{constr-Lie}\cite{BN}
Let $A$ be a vector space with a binary operation $\circ$. Define a binary operation on $A[t,t^{-1}]:=A\otimes \bfk[t,t^{-1}]$ by
   \vspace{-.1cm}
    \begin{eqnarray}\label{infLie1}
        [at^i, bt^j]=i(a\circ b)t^{i+j-1}-j(b\circ a)t^{i+j-1},\;\;\;
        a,b\in A, i, j\in\mathbb{Z},
      \vspace{-.1cm}
    \end{eqnarray}
    where $a t^i:=a\otimes t^{i}$. Then $(A[t,t^{-1}], [\cdot,\cdot])$
    is a Lie algebra if and only if $(A, \circ)$ is a Novikov algebra.
    \end{theorem}
\noindent Note that $t^i\diamond t^j=it^{i+j-1}$ for all $i$, $j\in\ZZ$ defines a right Novikov algebra structure on ${\bf k}[t,t^{-1}]$. Such construction of Lie algebras can be interpretated from the view of operads. Let $\mathcal{P}$ be a binary quadratic operad and $\mathcal{P}^{\textup{!}}$ be its operadic Koszul dual. By~\cite[Corollary~2.29]{GK} (see also~\cite[Proposition~7.6.5]{LV}), the tensor product of a $\mathcal{P}$-algebra with a $\mathcal{P}^{\textup{!}}$-algebra is naturally endowed with a Lie algebra structure, giving a potentially very general procedure of constructing Lie algebras.
Note that, by the result in \cite{Dz}, the Koszul dual of the operad of Novikov algebras is the operad of right Novikov algebras. Therefore, Novikov algebra affinization has
the operadic interpretation. It should be pointed out that the fact that the Lie bracket defined by Eq. (\ref{infLie1}) defining a Lie algebra characterizes a Novikov algebra is not trivial, since it has no good operadic interpretation, as far as we know.
The construction of Lie algebras from Novikov algebras and right Novikov algebras and the affinization of Novikov algebras were lifted to  the context of bialgebras  in \cite{HBG}. It was shown that there is a natural completed Lie bialgebra structure on the tensor product of a Novikov bialgebra and a quadratic $\ZZ$-graded right Novikov algebra. Moreover, when the quadratic $\ZZ$-graded right Novikov algebra is a specific one defined on the vector space ${\bf k}[t,t^{-1}]$, the definition of  Novikov bialgebras can be derived from the definition of  completed Lie bialgebras \cite{HBG}. Note that the operads of perm algebras and pre-Lie algebras are the Koszul dual each other and hence there is also a Lie algebra structure on the tensor product of a perm algebra and a pre-Lie algebra. The affinizations of perm algebras and pre-Lie algebras were introduced in \cite{LZB}. Such construction and affinizations were also extended  to the context of bialgebras in \cite{LZB}.

Note that there is a Novikov algebra associated to a pre-Novikov algebra and the commutator of a pre-Lie algebra is a Lie algebra. Therefore, there should exist a natural construction of pre-Lie algebras from pre-Novikov algebras and right Novikov algebras. Note that  a right Novikov algebra is a special right Novikov dialgebra \cite{SK,XBH}. In fact, we show that there is a natural pre-Lie algebra structure on the tensor product of a pre-Novikov algebra and a right Novikov dialgebra. In particular,  when the right Novikov dialgebra is taken as the one on ${\bf k}[t,t^{-1}]$ given in Theorem \ref{constr-Lie}, we present an affinization of pre-Novikov algebras, i.e. the tensor product of a pre-Novikov algebra and ${\bf k}[t,t^{-1}]$ with a special binary operation being a pre-Lie algebra characterizes the pre-Novikov algebra (see Theorems \ref{di-pre-L} and \ref{t5}). Motivated by the results in \cite{HBG, LZB}, the goal of this paper is to lift the construction of pre-Lie algebras from pre-Novikov algebras and right Novikov dialgebras and the affinization of pre-Novikov algebras to the level of bialgebras.

Note that the theory of pre-Novikov bialgebras was developed in  \cite{LH}. Based on the the results in \cite{HBG, LZB}, we need to construct a pre-Lie bialgebra from a finite-dimensional pre-Novikov bialgebra by a right Novikov dialgebra equipped with a special nondegenerate bilinear form. In the finite-dimensional case,  a pre-Lie bialgebra is equivalent to a para-K$\ddot{\rm{a}}$hler Lie algebra which gives a quadratic pre-Lie algebra \cite{Bai1}, and a pre-Novikov bialgebra is equivalent to a double construction of quasi-Frobenius Novikov algebras which gives a quadratic pre-Novikov algebra \cite{LH}. Therefore, in the finite-dimensional case, we need to consider whether the tensor product of a double construction of quasi-Frobenius Novikov algebras
and a right Novikov dialgebra with a special nondegenerate bilinear form can be endowed with a natural para-K$\ddot{\texttt{a}}$hler Lie algebra structure. Note that a para-K$\ddot{\rm{a}}$hler Lie algebra gives a quadratic pre-Lie algebra and  a double construction of quasi-Frobenius Novikov algebras gives a quadratic pre-Novikov algebra.
We should consider whether there is a natural quadratic pre-Lie algebra structure on the tensor product of a quadratic pre-Novikov algebra and a right Novikov dialgebra with a special nondegenerate bilinear form. We show that if the product of the bilinear form on the pre-Novikov algebra and the bilinear form on the right Novikov dialgebra defines a
quadratic pre-Lie algebra structure on the tensor product, then the right Novikov dialgebra with a special nondegenerate bilinear form should be a quadratic right Novikov algebra.
Note that this situation  is totally different from those in \cite{HBG, LZB}. Based on this observation, we  focus on constructing  a completed pre-Lie bialgebra from a finite-dimensional pre-Novikov bialgebra and a quadratic $\ZZ$-graded right Novikov algebra.

 %  On one hand, it is showed that there is a Novikiov algebra associated to a pre-Novikov algebra in\cite{HBG}. And pre-Lie algebras are Lie-admissible algebras
%On the other hand, by the study of operad,
%In general, the affinization of a Novikov algebra naturally defines a Lie algebra by a $\ZZ$-graded right Novikov algebra\cite{HBG}. Furthermore, there is a completed Lie bialgebra structure on the tensor product of a finite-dimensional Novikov bialgebra and a quadratic $\ZZ$-graded right Novikov algebra, which actually presents an afinization characterization of Novikov bialgebra\cite{HBG}.  It is that there is a Novikiov algebra associated to a pre-Novikov algebra and pre-Lie algebras are Lie-admissible algebras. Motivated by this, it is natural to construct a pre-Lie algebra from a pre-Novikov algebra by a $\ZZ$-graded right Novikov algebra. Note that the theory of pre-Novikov bialgebras was developed in \cite{LH}. Therefore, it is also naturally construct a completed pre-Lie bialgebra from a finite-dimensional pre-Novikov bialgebra by a quadratic $\ZZ$-graded right Novikov algebra.

This paper is organized as follows. In Section 2, we recall the definitions of pre-Novikov (co)algebras, right Novikov algebras and right Novikov dialgebras and  introduce the definitions of completed pre-Lie coalgebras and completed right Novikov co-dialgebras. We show that there is a pre-Lie algebra structure on the tensor product of a pre-Novikov algebra and  a $\ZZ$-graded right Novikov dialgebra.
Moreover, if we take the right Novikov dialgebra as a special right Novikov algebra on the vector space of Laurent polynomials,  the tensor product of a pre-Novikov algebra and this Novikov algebra being a pre-Lie algebra characterizes this pre-Novikov algebra, which is called the affinization of pre-Novikov algebras. Dually, we show that there is a completed pre-Lie coalgebra structure on the tensor product of a pre-Novikov coalgebra and a completed right Novikov co-dialgebra, and we also give an affinization of pre-Novikov coalgebras.

In Section 3, we first provide the reason why we choose a quadratic right Novikov algebra instead of a right Novikov dialgebra with a special bilinear form in the finite dimensional case. We show that there exists a completed pre-Lie bialgebra structure on the tensor product of a finite-dimensional pre-Novikov bialgebra and a quadratic $\ZZ$-graded right Novikov algebra. In the special case when the quadratic $\ZZ$-graded right Novikov algebra is a special one on the vector space of Laurent polynomials, we give a characterization of the pre-Novikov bialgebra by the affinization.

In Section 4, we recall the definitions of pre-Novikov Yang-Baxter equation in  a pre-Novikov algebra and  $S$-equation in a  pre-Lie algebra. We construct symmetric completed  solutions of  the $S$-equation in the corresponding pre-Lie algebra from symmetric solutions of the pre-Novikov Yang-Baxter equation in a pre-Novikov algebra. Conversely, through a special quadratic $\ZZ$-graded right Novikov algebra, we can obtain symmetric solutions of  the pre-Novikov Yang-Baxter equation in  the corresponding pre-Novikov algebra  from some symmetric completed  solutions of  the $S$-equation in the pre-Lie algebra.

%Then for a pre-Novikov algebra $A$ and a quadratic $\ZZ$-graded right Novikov algebra $B$, there is a construction of symmetric solutions of the S-equation in the pre-Lie algebra $A\otimes B$ from the symmetric solutions of the pre-Novikov Yang-Baxter equation in $A$.

\smallskip
\noindent
{\bf Notations.}
Throughout this paper, we fix a base field ${\bf k}$ of characteristic $0$. All vector spaces and algebras are over $\bf k$. The identity map is denoted by $\id$ and the set of integers is denoted by $\ZZ$. In this paper, $(A, \vartriangleleft,\vartriangleright)$ is assumed to be a finite-dimensional pre-Novikov algebra. Additionally, for a $\ZZ$-graded vector space $B=\oplus_{i\in \ZZ} B_i$, we assume that each homogeneous component $B_i$ is finite-dimensional.
Let $A$ be a vector space with a binary operation $\circ$.  Define linear maps
$L_{\circ}, R_{\circ}:A\rightarrow {\rm End}_{\bf k}(A)$ by
\begin{eqnarray*}
L_{\circ}(a)b:=a\circ b,\;\; R_{\circ}(a)b:=b\circ a, \;\;\;a, b\in A.
\end{eqnarray*}
Let
$$\tau:A\otimes A\rightarrow A\otimes A,\quad a\otimes b\mapsto b\otimes a,\;\;\;\; a,b\in A,$$
be the flip operator.

\delete{
Throughout this paper, all algebras are finite-dimensional and over a field $\mathbf{k}$ of characteristic zero. We give some notation as follows. Let $V,W$ be vector spaces.\\
\hspace*{4mm}(a) Let $(V,\circ)$ be a vector space with a binary operation $\circ$. We define the left action and right action by $L_\circ :V\rightarrow \text{End}_{\bf k}(V)$ and $R_\circ :V\rightarrow \text{End}_{\bf k}(V)$ respectively,
$$L_\circ (x)y=x\circ y=R_\circ (y)x,\quad x,y\in V.$$
\hspace*{4mm}(b) The exchanging operator $\tau:V\otimes V\rightarrow V\otimes V $ is defined by
\begin{align*}
\tau(x\otimes y)=y\otimes x,\quad  x,y\in V.
\end{align*}}

\section{Pre-Lie (co)algebras from pre-Novikov (co)algebras and right-Novikov (co)dialgebras}
In this section, we show that there is a pre-Lie algebra structure on the tensor product of a pre-Novikov algebra and a right Novikov dialgebra. Moreover, if we take the right Novikov dialgebra as a special right Novikov algebra on the vector space of Laurent polynomials, the tensor product of a pre-Novikov algebra and this right Novikov algebra being a pre-Lie algebra characterizes this pre-Novikov algebra, which is called an affinization of pre-Novikov algebras. Dually, we  show that there is a completed pre-Lie coalgebra structure on the tensor product of a pre-Novikov coalgebra and a completed right Novikov co-dialgebra and also give an affinization of pre-Novikov coalgebras.

\subsection{Pre-Lie algebras from pre-Novikov algebras and right-Novikov dialgebras}
Recall that a {\bf pre-Lie algebra} is a vector space $A$ with a binary operation $\circ :A\otimes A \rightarrow A$ satisfying
\begin{align}
(a\circ b)\circ c-a\circ (b\circ c)&=(b\circ a)\circ c-b\circ (a\circ c),\quad a,b,c\in A.
\end{align}
Denote it by $(A,\circ)$. If the binary operation $\circ$ also satisfies
\begin{align}
(a\circ b)\circ c=(a\circ c)\circ b, \quad a,b,c\in A,
\end{align}
then we call $(A,\circ)$ a {\bf Novikov algebra}.

 Let $(A,\circ )$ be a pre-Lie algebra. Then $A$ is a Lie algebra under the bracket operation $[a,b]=a\circ b-b\circ a$ for all $a$, $b\in A$, which is denoted by $(\mathcal{G}(A),[\cdot,\cdot])$. $(\mathcal{G}(A),[\cdot,\cdot])$ is called the {\bf sub-adjacent Lie algebra } of $(A,\circ )$. We also call $(A,\circ )$ the {\bf compatible pre-Lie algebra structure } on  the Lie algebra $(\mathcal{G}(A),[\cdot,\cdot])$.

\begin{definition}\label{pNov}\cite{HBG}
A {\bf pre-Novikov algebra} is a vector space $A$ with binary operations $\vartriangleleft,\vartriangleright: A\otimes A\rightarrow A$  satisfying
\begin{align}
a\vartriangleright (b\vartriangleright c)&=(a\circ b)\vartriangleright c+b\vartriangleright(a\vartriangleright c)-(b\circ a)\vartriangleright c,\label{pN-1}\\
a\vartriangleright(b\vartriangleleft c)&=(a\vartriangleright b)\vartriangleleft c+b\vartriangleleft(a\circ c)-(b\vartriangleleft a)\vartriangleleft c,\\
(a\circ b)\vartriangleright c&=(a\vartriangleright c)\vartriangleleft b,\label{e20}\\
(a\vartriangleleft b)\vartriangleleft c&=(a\vartriangleleft c)\vartriangleleft b, \label{pN-4} \quad a,b,c\in A,
\end{align}
where $a\circ b:=a\vartriangleleft b+a\vartriangleright b$. Denote it by $(A,\vartriangleleft,\vartriangleright)$.
\end{definition}

\begin{proposition}\cite[Proposition 3.33]{HBG}\label{an}
Let $(A,\vartriangleleft,\vartriangleright)$ be a pre-Novikov algebra. The binary operation
\begin{align}
a\circ b:=a\vartriangleleft b+a\vartriangleright b, \quad  a,b\in A, \label{e13}
\end{align}
defines a Novikov algebra $(A,\circ)$, which is called the \textbf{associated Novikov algebra} of $(A,\vartriangleleft,\vartriangleright)$.
\delete{Furthermore, $(A,L_{\vartriangleright},R_{\vartriangleleft})$ is a representation of $(A,\circ)$. Conversely, let $A$ be a vector space equipped with binary operations $\vartriangleleft$ and $\vartriangleright$. If $(A, \circ)$ defined by Eq. \eqref{e13} is a Novikov algebra and $(A,L_{\vartriangleright},R_{\vartriangleleft})$ is a representation of $(A,\circ)$, then $(A,\vartriangleleft,\vartriangleright)$ is a pre-Novikov algebra.}
\end{proposition}

%For the rest of this paper, $(A, \vartriangleleft,\vartriangleright)$ is assumed to be a finite-dimensional pre-Novikov algebra.

Next, we recall the definitions of  right Novikov dialgebras and right Novikov algebras.
\begin{definition}\cite{KS,XBH}
A {\bf right Novikov dialgebra} $(B,\dashv,\vdash)$ is a vector space $B$ equipped with two binary operations $\dashv$, $\vdash: B\otimes B\longrightarrow B$ satisfying the following compatibility conditions
\begin{align}
&x \vdash (y\vdash z)=y\vdash (	x\vdash z),\;\;	x\vdash (y\dashv z)=y\dashv (x\dashv  z),\label{RN-di1}\\
&(x\vdash y -x\dashv y)\vdash z=0,\;\;x \dashv(y\vdash z-y\dashv z)=0,\label{RN-di2}\\
&x \vdash (y\vdash z-z\dashv y)=(x\vdash y)\vdash z-(x\vdash z)\dashv y,\label{RN-di3}\\
&x \dashv ( y\vdash z -z\dashv y )=(x \dashv y) \dashv z-(x \dashv z) \dashv y,\quad x, y,z \in B.\label{RN-di4}
\end{align}
\end{definition}

\begin{definition}
A {\bf right Novikov algebra} $(B, \diamond)$ is a vector space $B$ equipped with a binary operation $\diamond: B\otimes B\longrightarrow B$ satisfying the following compatibility conditions
\begin{align}
&(x\diamond y)\diamond z-x\diamond(y\diamond z)=(x\diamond z)\diamond y-x\diamond (z\diamond y),\label{e18}\\
&x\diamond (y\diamond z)=y\diamond(x\diamond z),\qquad x,y,z\in B.\label{e19}
\end{align}
\end{definition}

\begin{remark}
If $(B,\dashv,\vdash)$ is a right Novikov dialgebra where $\dashv=\vdash$, then $(B,\diamond:=\dashv=\vdash)$ is a right Novikov algebra. Therefore, right Novikov algebras are special right Novikov dialgebras.
\end{remark}

\begin{definition}\cite{HBG}
A {\bf $\ZZ$-graded right Novikov algebra} is a right Novikov algebra $(B,\diamond )$ with a linear decomposition
    $B=\oplus_{i\in \ZZ} B_i$ such that  $B_i \diamond B_j\subset B_{i+j}$ for all $i, j\in \ZZ$.

A {\bf $\ZZ$-graded pre-Lie algebra}  is a pre-Lie algebra $(L,\circ)$ with a linear decomposition $L=\oplus_{i\in \ZZ} L_i$ such that $L_i\circ L_j\subset L_{i+j}$ for all $i, j\in \ZZ$.
\end{definition}

\begin{example}\label{rN-ex}\cite{HBG}
Let $B={\bf k}[t,t^{-1}]$ be the vector space of Laurent polynomials  with a binary operation $\diamond$ given by
$$t^i\diamond t^j\coloneqq \frac{d}{dt}(t^i)t^j=it^{i+j-1},\quad i, j\in \mathbb{Z}.$$
Then $(B,\diamond)$ is a right Novikov algebra. Moreover, let $B_i={\bf k}t^{i+1}$ for each $i\in \mathbb{Z}$. Then
    $(B=\oplus_{i\in \mathbb{Z}}B_i, \diamond)$ is a $\ZZ$-graded
    right Novikov algebra.
\end{example}

Next, we present a construction of pre-Lie algebras from pre-Novikov algebras and right Novikov dialgebras.
\begin{theorem}\label{di-pre-L}
Let $(A,\vartriangleleft,\vartriangleright)$ be a pre-Novikov algebra and $(B,\dashv,\vdash)$ be a  right Novikov dialgebra. Define a binary operation $\circ$ on $A\otimes B$ by
	\begin{flalign}
		(a\otimes x) \circ (b\otimes y):=a\vartriangleright b\otimes x\vdash y-b\vartriangleleft a \otimes y\dashv x,\quad a,b\in A,x,y\in B.\label{def-di-pL}
	\end{flalign}
	Then $(A\otimes B,\circ)$ is a pre-Lie algebra, which is called \textbf{the induced pre-Lie algebra} from $(A,\vartriangleleft,\vartriangleright)$ and $(B,\dashv,\vdash)$.
\end{theorem}
\begin{proof}
By Eq.~\eqref{def-di-pL}, we have
%\begin{flalign*}
%(a\otimes x \circ b\otimes y)\circ c\otimes z&=(a \vartriangleright b)\vartriangleright c\otimes(x\vdash y)\vdash z-c\vartriangleleft(a\vartriangleright b)\otimes z\dashv(x\vdash y)\\
%&\quad -(b\vartriangleleft a)\vartriangleright c\otimes (y\dashv x)\vdash z+c\vartriangleleft(b\vartriangleleft a)\otimes z\dashv(y\dashv x),\\
%a\otimes x \circ (b\otimes y\circ c\otimes z)&=a\vartriangleright (b\vartriangleright c)\otimes x\vdash(y\vdash z)-(b\vartriangleright c)\vartriangleleft a\otimes (y\vdash z)\dashv x\\
%& \quad -a\vartriangleright(c\vartriangleleft b)\otimes x\vdash(z\dashv y)+(c\vartriangleleft b)\vartriangleleft a\otimes (z\dashv y)\dashv x,\\
%(b\otimes y\circ a\otimes x)\circ c\otimes z&=(b\vartriangleright a)\vartriangleright c\otimes (y\vdash x)\vdash z-c\vartriangleleft(b\vartriangleright a)\otimes z\dashv(y\vdash x)\\
%&\quad -(a\vartriangleleft b)\vartriangleright c\otimes (x\dashv y)\vdash z+c\vartriangleleft(a\vartriangleleft b)\otimes z\dashv(x\dashv y),\\
%b\otimes y\circ (a\otimes x\circ c\otimes z)&=b\vartriangleright(a\vartriangleright c)\otimes y\vdash(x\vdash z)-(a\vartriangleright c)\vartriangleleft b\otimes (x\vdash z)\dashv y\\
%&\quad-b\vartriangleright(c\vartriangleleft a)\otimes y\vdash(z\dashv x)+(c\vartriangleleft a)\vartriangleleft b\otimes (z\dashv x)\dashv y.
%\end{flalign*}
\begin{flalign*}
	&((a\otimes x) \circ (b\otimes y))\circ (c\otimes z)-(a\otimes x) \circ ((b\otimes y)\circ (c\otimes z))\\
	&\quad\quad-((b\otimes y)\circ (a\otimes x))\circ (c\otimes z)+(b\otimes y)\circ ((a\otimes x)\circ (c\otimes z))\\
	&\quad =(a \vartriangleright b)\vartriangleright c\otimes(x\vdash y)\vdash z-c\vartriangleleft(a\vartriangleright b)\otimes z\dashv(x\vdash y)-(b\vartriangleleft a)\vartriangleright c\otimes (y\dashv x)\vdash z\\
	&\quad\quad +c\vartriangleleft(b\vartriangleleft a)\otimes z\dashv(y\dashv x)
	-a\vartriangleright (b\vartriangleright c)\otimes x\vdash(y\vdash z)+(b\vartriangleright c)\vartriangleleft a\otimes (y\vdash z)\dashv x\\
	&\quad\quad +a\vartriangleright(c\vartriangleleft b)\otimes x\vdash(z\dashv y)
	-(c\vartriangleleft b)\vartriangleleft a\otimes (z\dashv y)\dashv x-(b\vartriangleright a)\vartriangleright c\otimes (y\vdash x)\vdash z\\
	&\quad\quad +c\vartriangleleft(b\vartriangleright a)\otimes z\dashv(y\vdash x)+(a\vartriangleleft b)\vartriangleright c\otimes (x\dashv y)\vdash z-c\vartriangleleft(a\vartriangleleft b)\otimes z\dashv(x\dashv y)\\
	&\quad\quad +b\vartriangleright(a\vartriangleright c)\otimes y\vdash(x\vdash z)
	-(a\vartriangleright c)\vartriangleleft b\otimes (x\vdash z)\dashv y-b\vartriangleright(c\vartriangleleft a)\otimes y\vdash(z\dashv x)\\
	&\quad\quad +(c\vartriangleleft a)\vartriangleleft b\otimes (z\dashv x)\dashv y.\\
\end{flalign*}

By Eqs.~\eqref{RN-di1}-\eqref{RN-di2}, we obtain
\begin{flalign*}	
	&x\vdash (y\vdash z)=y\vdash (x\vdash z),\quad z\dashv (y\vdash x)=y\vdash (z\dashv x)=z\dashv (y\dashv x), \\
	&x\vdash (z\dashv y)=z\dashv (x\dashv y)=z\dashv (x\vdash y),  \\
	&(y\vdash x )\vdash z=(y\dashv x)\vdash z,\quad
	(x\dashv y)\vdash z=(x\vdash y)\vdash z.
\end{flalign*}
Then applying the above equalities, we get
\begin{flalign*}	
	&((a\otimes x) \circ (b\otimes y))\circ (c\otimes z)-(a\otimes x) \circ ((b\otimes y)\circ (c\otimes z))\\
&\quad\quad-((b\otimes y)\circ (a\otimes x))\circ (c\otimes z)+(b\otimes y)\circ ((a\otimes x)\circ (c\otimes z))\\
	&\quad = (b\vartriangleright (a\vartriangleright c)-a\vartriangleright(b\vartriangleright c))\otimes x\vdash(y\vdash z)\\
	&\quad\quad +(c\vartriangleleft(b\vartriangleleft a)+c\vartriangleleft(b\vartriangleright a)-b\vartriangleright(c\vartriangleleft a))\otimes z\dashv(y\vdash x)\\
	&\quad\quad+(-c\vartriangleleft(a\vartriangleright b)+a\vartriangleright(c\vartriangleleft b)-c\vartriangleleft(a\vartriangleleft b))\otimes z\dashv(x\vdash y)\\
	&\quad\quad -((b \vartriangleleft a)\vartriangleright c+(b\vartriangleright a)\vartriangleright c)\otimes (y\vdash x )\vdash z\\
    &\quad\quad +((a\vartriangleright b)\vartriangleright c+(a\vartriangleleft b)\vartriangleright c)\otimes (x\vdash y)\vdash z\\
    &\quad\quad +(c\vartriangleleft b)\vartriangleleft a\otimes((z\dashv x)\dashv y-(z\dashv y)\dashv x)\\
	&\quad\quad +(b \vartriangleright c)\vartriangleleft a\otimes (y\vdash z)\dashv x\\
	&\quad\quad-(a\vartriangleright c)\vartriangleleft b\otimes (x\vdash z)\dashv y.
\end{flalign*}

By Eqs.~\eqref{RN-di3}-\eqref{RN-di4}, we obtain
\begin{flalign*}
(z \dashv x)\dashv y-(z\dashv y)\dashv x=z\dashv(x\vdash y)-z\dashv(y\dashv x),\\
(y\vdash z)\dashv x=(y\vdash x)\vdash z -y\vdash(x \vdash z)+y\vdash (z\dashv x),\\
(x\vdash z)\dashv y=(x\vdash y)\vdash z-x\vdash(y\vdash z)+x\vdash(z\dashv y).
\end{flalign*}
Then applying the above equalities, we get
\begin{flalign*}	
		&((a\otimes x) \circ (b\otimes y))\circ (c\otimes z)-(a\otimes x) \circ ((b\otimes y)\circ (c\otimes z))\\
	&\quad\quad-((b\otimes y)\circ (a\otimes x))\circ (c\otimes z)+(b\otimes y)\circ ((a\otimes x)\circ (c\otimes z))\\
		&\quad = (b\vartriangleright (a\vartriangleright c)-a\vartriangleright(b\vartriangleright c))\otimes x\vdash(y\vdash z)\\
		&\quad\quad +(c\vartriangleleft(b\vartriangleleft a)+c\vartriangleleft(b\vartriangleright a)-b\vartriangleright(c\vartriangleleft a))\otimes z\dashv(y\vdash x)\\
		&\quad\quad(-c\vartriangleleft(a\vartriangleright b)+a\vartriangleright(c\vartriangleleft b)-c\vartriangleleft(a\vartriangleleft b))\otimes z\dashv(x\vdash y)\\
		&\quad\quad -((b \vartriangleleft a)\vartriangleright c+(b\vartriangleright a)\vartriangleright c)\otimes (y\vdash x )\vdash z\\
		&\quad\quad +((a\vartriangleright b)\vartriangleright c+(a\vartriangleleft b)\vartriangleright c)\otimes (x\vdash y)\vdash z\\
		&\quad\quad +(c\vartriangleleft b)\vartriangleleft a\otimes(z\dashv(x\vdash y)-z\dashv(y\dashv x))\\
		&\quad\quad +(b \vartriangleright c)\vartriangleleft a\otimes ((y\vdash x)\vdash z -y\vdash(x \vdash z)+y\vdash (z\dashv x))\\
		&\quad\quad-(a\vartriangleright c)\vartriangleleft b\otimes ((x\vdash y)\vdash z-x\vdash(y\vdash z)+x\vdash(z\dashv y))\\
		&\quad = (b\vartriangleright (a\vartriangleright c)-a\vartriangleright(b\vartriangleright c)+(a\vartriangleright c )\vartriangleleft b  -(b\vartriangleright c)\vartriangleleft a)\otimes x\vdash(y\vdash z)\\
		&\quad\quad +(c\vartriangleleft(b\vartriangleleft a+b\vartriangleright a)-b\vartriangleright(c\vartriangleleft a)-(c\vartriangleleft b)\vartriangleleft a+(b\vartriangleright c)\vartriangleleft a)\otimes z\dashv(y\vdash x)\\
		&\quad\quad(-c\vartriangleleft(a\vartriangleright b+a\vartriangleleft b)+a\vartriangleright(c\vartriangleleft b)+(c\vartriangleleft b)\vartriangleleft a-(a\vartriangleright c)\vartriangleleft b)\otimes z\dashv(x\vdash y)\\
		&\quad\quad -((b \vartriangleleft a+b\vartriangleright a)\vartriangleright c-(b\vartriangleright c)\vartriangleleft a)\otimes (y\vdash x )\vdash z\\
		&\quad\quad +((a\vartriangleright b+a\vartriangleleft b)\vartriangleright c-(a\vartriangleright c)\vartriangleleft b)\otimes (x\vdash y)\vdash z\\
		&=0.
\end{flalign*}

Therefore, $(A\otimes B,\circ)$ is a pre-Lie algebra.
\end{proof}
\begin{remark}
 This theorem was observed by Chengming Bai and Guilai Liu. We
 thank them for communicating it to us.
\end{remark}

\begin{theorem}\label{t5}
Let $A$ be a vector space and $\vartriangleleft$, $\vartriangleright:A \otimes A\rightarrow A$ be two binary operations. Assume that $(B,\diamond)=({\bf k}[t,t^{-1}],\diamond)$ is the $\mathbb{Z}$-graded right Novikov algebra given in Example \ref{rN-ex}.
Then $(A\otimes B,\circ)$ is a $\mathbb{Z}$-graded pre-Lie algebra defined by Eq. (\ref{def-di-pL}) if and only if $(A,\vartriangleleft,\vartriangleright)$ is a pre-Novikov algebra.
\end{theorem}

\begin{proof}
If $(A,\vartriangleleft,\vartriangleright)$ is a pre-Novikov algebra, then $(A\otimes B,\circ)$ is a pre-Lie algebra by  Theorem \ref{di-pre-L}. Since $(B=\oplus_{i\in \ZZ}B_i, \diamond)$ is a $\ZZ$-graded right Novikov algebra,  it follows from $(A\otimes B_i)\circ (A\otimes B_j)\subset A\otimes B_{i+j}$ that $(A\otimes B, \circ)=(\oplus_{i\in \ZZ}(A\ot B_i),\circ)$ is a $\ZZ$-graded pre-Lie algebra. Next, we prove the ``only if" part.

% Meanwhile, $(A\otimes B,\circ)$ is a $\ZZ$-graded pre-Lie algebra since $(B,\diamond)$ is a $\mathbb{Z}$-graded right Novikov algebra.
Note that by Eq. (\ref{def-di-pL}), we have
\begin{flalign*}
(a\otimes t^i)\circ (b\otimes t^j)=a\vartriangleright b\otimes i t^{i+j-1}-b\vartriangleleft a\otimes jt^{i+j-1},\quad a, b\in A,i,j\in \mathbb{Z}.
\end{flalign*}
 If $(A\otimes B,\circ)$ is a pre-Lie algebra, for all $a, b,c\in A$, $i,j,k\in \mathbb{Z}$, we obtain
\begin{flalign*}
&((a\otimes t^i) \circ (b\otimes t^j))\circ (c\otimes t^k)-(a\otimes t^i )\circ ((b\otimes t^j)\circ (c\otimes t^k))-((b\otimes t^j) \circ (a\otimes t^i)  )\circ( c\otimes t^k)\\
&\quad\quad +(b\otimes t^j) \circ ((a\otimes t^i)  \circ (c\otimes t^k))\\
&\quad =(a\vartriangleright b)\vartriangleright c\otimes i(i+j-1)t^s-c\vartriangleleft (a\vartriangleright b)\otimes ik t^s-(b\vartriangleleft a)\vartriangleright c\otimes j(i+j-1)t^s\\
&\quad\quad +c\vartriangleleft (b\vartriangleleft a)\otimes jk t^s-a\vartriangleright(b\vartriangleright c)\otimes ij t^s +(b\vartriangleright c )\vartriangleleft a\otimes j(j+k-1) t^s\\
&\quad\quad +a\vartriangleright(c\vartriangleleft b)\otimes ik t^s-(c\vartriangleleft b)\vartriangleleft a\otimes k(j+k-1) t^s-(b\vartriangleright a )\vartriangleright c\otimes j(i+j-1) t^s\\
&\quad\quad +c\vartriangleleft(b\vartriangleright a)\otimes jk t^s+(a\vartriangleleft b)\vartriangleright c\otimes i(i+j-1) t^s-c\vartriangleleft(a\vartriangleleft b)\otimes ik t^s\\
&\quad\quad +b\vartriangleright(a\vartriangleright c)\otimes ij t^s-(a\vartriangleright c)\vartriangleleft b\otimes i(i+k-1) t^s-b\vartriangleright(c\vartriangleleft a)\otimes jk t^s\\
&\quad\quad +(c\vartriangleleft a)\vartriangleleft b\otimes k(i+k-1) t^s\\
&\quad=0,
\end{flalign*}
where $s=i+j+k-2$.\\

Comparing the coefficients of $i(i-1)t^s$, $k(k-1)t^s$, $ijt^s$ and $jkt^s$ in the above equality, we obtain
\begin{flalign*}
&(a\vartriangleright b)\vartriangleright c+(a\vartriangleleft b)\vartriangleright c-(a\vartriangleright c)\vartriangleleft b=0,\\
&(c\vartriangleleft a)\vartriangleleft b-(c\vartriangleleft b)\vartriangleleft a=0,\\
&(a\vartriangleright b)\vartriangleright c-(b\vartriangleleft a)\vartriangleright c-a\vartriangleright (b\vartriangleright c)-(b\vartriangleright a)\vartriangleright c+(a\vartriangleleft b)\vartriangleright c+b\vartriangleright(a\vartriangleright c)=0,\\
&c\vartriangleleft (b\vartriangleleft a)+(b\vartriangleright c)\vartriangleleft a-(c\vartriangleleft b)\vartriangleleft a+c\vartriangleleft (b\vartriangleright a)-b\vartriangleright(c\vartriangleleft a)=0.
\end{flalign*}
Therefore $(A,\vartriangleleft,\vartriangleright)$ is a pre-Novikov algebra.

The proof is completed.
\end{proof}
\subsection{Completed pre-Lie coalgebras from pre-Novikov coalgebras and completed right-Novikov co-dialgebras}
We first recall the definition of pre-Novikov coalgebras.
\begin{definition}\cite{LH}\label{co1}
A \textbf{pre-Novikov coalgebra} is a vector space $A$ with linear maps $\alpha, \beta:A\rightarrow A\otimes A$ satisfying
%$R_1=0,R_2=0,R_3=0,R_4=0,$
%where $R_i:A\rightarrow A\otimes A\otimes A$ are linear maps defined by $(i=1,2,3,4)$
\begin{eqnarray}
&&\label{cob1}(\alpha\otimes \id)\alpha(a)+(\tau\otimes \id)(\id\otimes \alpha)\beta(a)-(\id\otimes(\alpha+\beta))\alpha(a)-(\tau\otimes \id)(\beta\otimes \id)\alpha(a)=0,\\
&&\label{cob2}(\id\otimes\beta)\beta(a)+(\tau\otimes \id)((\alpha+\beta)\otimes \id)\beta(a)-((\alpha+\beta)\otimes \id)\beta(a)-(\tau\otimes \id)(\id\otimes \beta)\beta(a)=0,\\
&&\label{cob3}(\id \otimes \tau)(\beta\otimes \id)\alpha(a)-((\alpha+\beta)\otimes \id)\beta(a)=0,\\
&&\label{cob4}(\id\otimes\tau)(\alpha\otimes \id)\alpha(a)-(\alpha\otimes \id)\alpha(a)=0,\quad  a\in A.
\end{eqnarray}
Denote it by $(A,\alpha,\beta)$.
\end{definition}

Let $V$ and $W$ be vector spaces. Suppose that $\phi:V\rightarrow W$ is a linear map. Then there is an induced dual linear map $\phi^*:W^*\rightarrow V^*$ defined by
\begin{align}
\langle \phi^*(f),v\rangle=\langle f, \phi(v)\rangle,\quad  v\in V,f\in W^*.\label{a1}
\end{align}
It is easy to see that
$(A,\alpha,\beta)$ is a pre-Novikov coalgebra if and only if $(A^*,\alpha^*,\beta^*)$ is a pre-Novikov algebra, if $A$ is finite-dimensional.

By \cite[Lemma 2.10]{HBG}, we need to recall the notion of completed tensor products. \delete{Hence we can obtain a coalgebra structure on the Laurent polynomials whose graded linear map dual is the right Novikov algebra of Laurent polynomials in Example \ref{rN-ex}.}

Let $C=\oplus_{i\in \ZZ} C_i$ and $D=\oplus_{j\in \ZZ} D_j$ be $\ZZ$-graded vector spaces. We call the {\bf completed tensor product} of $C$ and $D$ to be the vector space
$$C\hatot  D\coloneqq \prod_{i,j\in\ZZ} C_i\ot D_j.
$$
If $C$ and $D$ are finite-dimensional, then $C\hatot  D$ is just the usual tensor product $C\otimes D$.

In general, an element in $C\hatot  D$ is an infinite  sum $\sum_{i,j\in\ZZ} \ts_{ij} $ with $\ts_{ij}\in C_i\ot D_j$. So $\ts_{ij}=\sum_\alpha c_{i\alpha}\ot d_{j\alpha}$ for pure tensors $c_{i\alpha}\ot d_{j\alpha}\in C_i\ot D_j$ with $\alpha$ in a finite index set.
Thus a general term of $C\hatot  D$ is a possibly infinite sum
\vspace{-.2cm}
\begin{equation}\label{eq:ssum}
 \sum_{i,j,\alpha} c_{i\alpha}\ot d_{j\alpha},
 \end{equation}
where $i,j\in \ZZ$ and $\alpha$ is in a finite index set (which might depend on $i,j$).

With these notations, for linear maps $f:C\to C'$ and $g:D\to D'$, define
$$ f\hatot g: C\hatot D \to C' \hatot D',
\quad \sum_{i,j,\alpha} c_{i\alpha}\ot d_{j\alpha}\mapsto  \sum_{i,j,\alpha} f(c_{i\alpha})\ot g(d_{j\alpha}).
\vspace{-.2cm}
$$
Also the twist map $\tau$ has its completion

$$ \widehat{\tau}: C\hatot C \to C\hatot C, \quad
\sum_{i,j,\alpha} c_{i\alpha} \ot d_{j\alpha} \mapsto \sum_{i,j,\alpha} d_{j\alpha}\ot c_{i\alpha}.
$$
Finally, we define a completed coproduct to be a linear map
$$ \Delta: C\to C\hatot C, \quad \Delta(a)
:=\sum_{i,j,\alpha} a_{1i\alpha}\ot a_{2j\alpha}.
$$
Then we have the well-defined map

$$ (\Delta \hatot \id)\Delta(a)=(\Delta \hatot \id)\Big(\sum_{i,j,\alpha} a_{1i\alpha}\ot a_{2j\alpha}\Big) :=\sum_{i,j,\alpha} \Delta(a_{1i\alpha})\ot a_{2j\alpha}
\in C\hatot C \hatot C.
$$

\begin{definition}
 A {\bf \complete right Novikov co-dialgebra} is a triple $(B,\da,\db)$, where $B$ is a vector space and
$\da,\db: B\rightarrow B\hatot B$ are linear maps satisfying for all $x\in B$
\begin{eqnarray}
	&&\label{RNov-Codialg-1} ( \id \hatot\db )\db(x)=(\tau \hatot \id)(\id \hatot\db )\db(x),(\id \hatot \da)\db(x)=(\tau \hatot \id)(\id \hatot\da )\da(x),\\
	&&\label{RNov-Codialg-2}(\db \hatot \id)\db(x)=(\da \hatot\id )\db(x),( \id\hatot\db )\da(x)=(\id \hatot\da )\da(x),\\
	&&\label{RNov-Codialg-3}(\id \hatot\db )\db(x)-(\id \hatot\tau )(\id \hatot\da(b) )\db(x)=(\db \hatot \id)\db(x)-(\id \hatot\tau )(\db \hatot \id)\da(x),\\
	&&\label{RNov-Codialg-4}(\id \hatot\db )\da(x)-(\id \hatot\tau )(\id \hatot\da(b) )\da(x)=(\da \hatot \id)\da(x)-(\id \hatot\tau )(\da \hatot \id)\da(x).
\end{eqnarray}
\end{definition}
\begin{remark}
	If $B$ is finite-dimensional, then $(B,\da,\db)$ is called a {\bf  right Novikov co-dialgebra}.
\end{remark}

\begin{definition}
\begin{enumerate}
\item \cite{HBG}  A {\bf \complete right Novikov coalgebra} is a pair $(B,\Delta)$, where $B=\oplus_{i\in \ZZ}B_i$ is a $\mathbb{Z}$-graded vector space and
    $\Delta: B\rightarrow B\hatot  B$ is a linear map satisfying
    \begin{eqnarray}
        &&\label{top-Nov-Coalg-1}(\Delta \hatot  \id)\Delta(x)-(\id \hatot  \widehat{\tau})(\Delta \hatot  \id)\Delta(x)=(\id \hatot  \Delta)\Delta(x)-(\id \hatot  \widehat{\tau})(\id \hatot  \Delta )\Delta(x),\\
        &&\label{top-Nov-Coalg-2}(\id \hatot  \Delta)\Delta(x)=(\widehat{\tau}\hatot  \id) (\id \hatot  \Delta)\Delta(x),\;\;\; x\in B.
    \end{eqnarray}
\item
A {\bf \complete pre-Lie coalgebra}
 is a pair $(L, \delta)$, where
$L=\oplus_{i\in \ZZ} L_i$
  is a $\ZZ$-graded vector space and $\delta: L\rightarrow L\hatot L$ is a linear map satisfying
    \begin{align}
    (\id\hatot \delta)\delta(a)-(\tau\hatot \id)(\id \hatot \delta)\delta(a)=(\delta\hatot \id)\delta(a)-(\tau\hatot \id)(\delta\hatot \id)\delta(a),  \;\; a\in L\label{co-pL}.
    \end{align}
\end{enumerate}
\end{definition}

\begin{remark}
When $B$ (resp. $L$)  is finite-dimensional, $(B, \Delta)$ (resp. $(L, \delta)$) is just a right Novikov coalgebra (resp. a pre-Lie coalgebra).
 If $(B,\da,\db)$ is a completed right Novikov co-dialgebra in which $\da=\db$, then $(B,\Delta:=\da=\db)$ is a completed right Novikov coalgebra.
\end{remark}

\begin{example}\mlabel{Laurent-coproduct}\cite{HBG}
Let $B\coloneqq {\bf k}[t,t^{-1}]$  with a linear map $\Delta: B\rightarrow B \hatot  B$ given by
\begin{eqnarray}\mlabel{eq:laurentco}
     \notag
\Delta(t^j)=\sum_{i\in \mathbb{Z}}(i+1)t^{-i-2}\otimes  t^{j+i} , j\in \mathbb{Z}.
\end{eqnarray}
Then $(B={\bf k}[t,t^{-1}], \Delta)$ is a \complete right Novikov coalgebra.
\end{example}

Let $C$ be a vector space and $D=\oplus_{j\in \ZZ} D_j$ be a $\ZZ$-graded vector space. For all $\sum_{i_1,\cdots,i_k}c_{i_1}\otimes \cdots \otimes c_{i_k}\in C\otimes \cdots \otimes C$ and $\sum_{j_1,\cdots,j_k,\alpha}d_{j_1\alpha}\otimes \cdots\otimes d_{j_k\alpha}\in D\hatot \cdots \hatot D$, we set
{\small \begin{eqnarray*}
\sum_{i_1,\cdots,i_k}c_{i_1}\otimes \cdots \otimes c_{i_k} \bullet \sum_{j_1,\cdots,j_k,\alpha}d_{j_1\alpha}\otimes \cdots\otimes d_{j_k\alpha}:=  \sum_{j_1,\cdots,j_k,\alpha}\sum_{i_1,\cdots,i_k}(c_{i_1}\otimes d_{j_1\alpha})\otimes \cdots \otimes(c_{i_k}\otimes d_{j_k\alpha}),
\end{eqnarray*}}
which belongs to $(C\otimes D)\hatot (C\otimes D)\hatot\cdots \hatot (C\otimes D)$.

Now we give the coalgebra version of Theorem \ref{di-pre-L}.
\begin{theorem}\label{tt6}
	Let $(A,\alpha,\beta)$ be a pre-Novikov coalgebra, $(B,\Delta_{\alpha},\Delta_{\beta})$ be a completed right Novikov co-dialgebra. Then $(L:=A\otimes B, \delta)$ is a \complete pre-Lie coalgebra, where the linear map $\delta:L\rightarrow L \hatot  L$ is given by
	\begin{eqnarray}\label{co-dipL}
		\delta(a\otimes x)=\beta(a)\bullet \db(x)-\widehat{\tau}\alpha(a)\bullet\widehat{\tau}\da (x),  \quad a\in A, x\in B.
	\end{eqnarray}
	\delete{To be more specific, we write $\alpha(a)=\sum_{[a]}
	a_{[1]}\otimes a_{[2]}$ and $\beta(a)=\sum_{(a)}a_{(1)}\ot a_{(2)}$ in the Sweedler notation.
	And $\da(x)
	=\sum _{i,j,\gamma} x_{[1]i\gamma}\ot x_{[2]j\gamma}$, $\db(x)
	=\sum _{i,j,\gamma} x_{(1)i\gamma}\ot x_{(2)j\gamma}$ as expression in Eq. \eqref{eq:ssum}, we set
	\begin{flalign}
		\alpha(a)\bullet \da(x)\coloneqq \sum_{[a]}\sum_{i,j,\gamma} (a_{[1]}\otimes x_{[1]i\gamma})\otimes (a_{[2]}\otimes x_{[2]j\gamma}),\label{1111}\\
		\beta(a)\bullet \db(x)\coloneqq \sum_{(a)}\sum_{i,j,\gamma} (a_{(1)}\otimes x_{(1)i\gamma})\otimes (a_{(2)}\otimes x_{(2)j\gamma}).\label{2222}
	\end{flalign}}
\end{theorem}
\begin{proof}
\delete{For
$\sum_{\ell}a'_\ell\otimes a''_\ell\otimes a'''_\ell\in A\otimes A\otimes A$ and
$\sum_{i,j,k,\alpha}x'_{i\alpha}\otimes x''_{j\alpha}\otimes x'''_{k\alpha}\in B\hatot  B\hatot  B$, we denote
\begin{eqnarray*}
	\Big(\sum_{\ell}a'_\ell\otimes a''_\ell\otimes a'''_\ell\Big)\bullet \Big(\sum_{i,j,k,\alpha }x'_{i\alpha}\otimes x''_{j\alpha}\otimes x'''_{k\alpha}\Big):=\sum_\ell \sum_{i,j,k,\alpha}(a'_\ell\otimes x'_{i\alpha})\otimes (a''_\ell\otimes x''_{j\alpha})\otimes (a'''_\ell\otimes  x'''_{k\alpha}).
\end{eqnarray*}}
Applying the definition of $\delta$, for all $a\in A$, $x\in B$, we obtain
{\small
	\begin{flalign*}
		&((\id\hatot\delta)\delta-(\tau\hatot \id)(\id \hatot \delta)\delta-(\delta\hatot \id)\delta+(\tau\hatot\id)(\delta\hatot \id)\delta)(a\ot x)\\
		&\quad=(\id \ot\beta ) \beta(a)\bullet (\id \hatot \db)\db(x)-(\id\ot \tau)(\id\ot \alpha)\beta(a) \bullet (\id\hatot\widehat{\tau})(\id\hatot\da)\db(x)\\
		&\quad\quad-(\id \ot\beta )\tau\alpha(a) \bullet(\id \hatot\db)\widehat{\tau}\da(x)+(\id\ot\tau)(\id\ot\alpha)\tau\alpha(a) \bullet(\id\hatot\widehat{\tau})(\id\hatot\da)\widehat{\tau}\da(x)\\
		&\quad\quad-(\tau\ot \id)(\id\ot \beta)\beta(a)\bullet(\widehat{\tau}\hatot \id)(\id\hatot\db)\db(x)\\
		&\quad\quad+(\tau\ot\id)(\id \ot\tau)(\id\ot\alpha)\beta(a)\bullet(\widehat{\tau}\hatot\id)(\id\hatot\widehat{\tau})(\id\hatot\da)\db(x)\\
		&\quad\quad+(\tau\ot\id)(\id\ot\beta)\tau\alpha(a)\bullet(\widehat{\tau}\hatot\id)(\id\hatot\db)\widehat{\tau}\da(x)\\
		&\quad\quad-(\tau\ot\id)(\id\ot\tau)(\id\ot\alpha)\tau\alpha(a)\bullet(\widehat{\tau}\hatot\id)(\id\hatot\widehat{\tau})(\id\hatot\da)\widehat{\tau}\da(x)\\
		&\quad\quad-(\beta\ot\id)\beta(a)\bullet(\db\hatot\id)\db(x)+(\tau\ot\id)(\alpha\ot\id)\beta(a)\bullet(\widehat{\tau}\hatot\id)(\da\hatot\id)\db(x)\\
		&\quad\quad+(\beta\ot\id)\tau\alpha(a)\bullet(\db\hatot\id)\widehat{\tau}\da(x)-(\tau\ot\id)(\alpha\ot\id)\tau\alpha(a)
		\bullet(\widehat{\tau}\hatot\id)(\da\hatot\id)\widehat{\tau}\da(x)\\
		&\quad\quad+(\tau\ot\id)(\beta\ot\id)\beta(a)\bullet(\widehat{\tau}\hatot\id)(\db\hatot\id)\db(x)-(\alpha\ot\id)\beta(a)\bullet(\da\hatot\id)\db(x)\\
		&\quad\quad-(\tau\ot\id)(\beta\ot\id)\tau\alpha(a)\bullet(\widehat{\tau}\hatot\id)(\db\hatot\id)\widehat{\tau}\da(x)+(\alpha\ot\id)\tau\alpha(a)
		\bullet(\da\hatot\id)\widehat{\tau}\da(x).
\end{flalign*}}
By Eqs. \eqref{RNov-Codialg-1}-\eqref{RNov-Codialg-4}, we obtain
{\small \begin{flalign*}
&(\id \hatot \db)\db(x)=(\widehat{\tau}\hatot\id)(\id\hatot\db)\db(x),\\
&(\id\hatot\widehat{\tau})(\id\hatot\da)\db(x)=(\db\hatot\id)\widehat{\tau}\da(x)=(\da\hatot\id)\widehat{\tau}\da(x),\\
&(\widehat{\tau}\hatot\id)(\id\hatot\widehat{\tau})(\id\hatot\da)\db(x)=(\widehat{\tau}\hatot\id)(\da\hatot\id)\widehat{\tau}\da(x)=(\widehat{\tau}\hatot\id)(\db\hatot\id)\widehat{\tau}\da(x),\\
&(\widehat{\tau}\hatot\id)(\da\hatot\id)\db(x)=(\widehat{\tau}\hatot\id)(\db\hatot\id)\db(x),\\
&(\id\hatot\db)\widehat{\tau}\da(x)=(\widehat{\tau}\hatot\id)(\db\hatot\id)\db(x)+(\widehat{\tau}\hatot\id)(\id\hatot\widehat{\tau})(\id\hatot\da)\db(x)-(\widehat{\tau}\hatot\id)(\id\hatot\db)\db(x),\\
&(\id\hatot\widehat{\tau})(\id\hatot\da)\widehat{\tau}\da(x)=( \widehat{\tau}\hatot\id)(\id\hatot \widehat{\tau})(\id\hatot \da)\widehat{\tau}\da(x)+(\widehat{\tau}\hatot\id)(\da \hatot\id)\widehat{\tau}\da(x)-(\db \hatot\id)\widehat{\tau}\da(x).
\end{flalign*}}

Then applying these equalities together with Eqs. (\mref{RNov-Codialg-1})-(\mref{RNov-Codialg-4}) and Eqs. (\mref{cob1})-(\mref{cob4}), we get
{\small
	\begin{flalign*}
		&((\id\hatot\delta)\delta-(\tau\hatot \id)(\id \hatot \delta)\delta-(\delta\hatot \id)\delta+(\tau\hatot\id)(\delta\hatot \id)\delta)(a\ot x)\\
		&\quad=(\id \ot\beta ) \beta(a)\bullet (\id \hatot \db)\db(x)-(\id\ot \tau)(\id\ot \alpha)\beta(a) \bullet (\id\hatot\widehat{\tau})(\id\hatot\da)\db(x)\\
		&\quad\quad-(\id \ot\beta )\tau\alpha(a) \bullet	((\widehat{\tau}\hatot\id)(\db\hatot\id)\db(x)+(\widehat{\tau}\hatot\id)(\id\hatot\widehat{\tau})(\id\hatot\da)\db(x)\\
    	&\quad\quad-(\widehat{\tau}\hatot\id)(\id\hatot\db)\db(x))	+(\id\ot\tau)(\id\ot\alpha)\tau\alpha(a) \bullet (( \widehat{\tau}\hatot\id)(\id\hatot \widehat{\tau})(\id\hatot \da)\widehat{\tau}\da(x)\\
    	&\quad\quad++(\widehat{\tau}\hatot\id)(\da \hatot\id)\widehat{\tau}\da(x)-(\db \hatot\id)\widehat{\tau}\da(x))\\
    	&\quad\quad-(\tau\ot \id)(\id\ot \beta)\beta(a)\bullet(\widehat{\tau}\hatot \id)(\id\hatot\db)\db(x)\\
		&\quad\quad+(\tau\ot\id)(\id \ot\tau)(\id\ot\alpha)\beta(a)\bullet(\widehat{\tau}\hatot\id)(\id\hatot\widehat{\tau})(\id\hatot\da)\db(x)\\
		&\quad\quad+(\tau\ot\id)(\id\ot\beta)\tau\alpha(a)\bullet((\db \hatot \id)\db(x)+(\id \hatot\tau )(\id \hatot\da(x) )\db(x)-(\id \hatot\db )\db(x))\\
		&\quad\quad-(\tau\ot\id)(\id\ot\tau)(\id\ot\alpha)\tau\alpha(a)\bullet(\widehat{\tau}\hatot\id)(\id\hatot\widehat{\tau})(\id\hatot\da)\widehat{\tau}\da(x)\\
		&\quad\quad-(\beta\ot\id)\beta(a)\bullet(\db\hatot\id)\db(x)+(\tau\ot\id)(\alpha\ot\id)\beta(a)\bullet(\widehat{\tau}\hatot\id)(\da\hatot\id)\db(x)\\
		&\quad\quad+(\beta\ot\id)\tau\alpha(a)\bullet(\db\hatot\id)\widehat{\tau}\da(x)-(\tau\ot\id)(\alpha\ot\id)\tau\alpha(a)
		\bullet(\widehat{\tau}\hatot\id)(\da\hatot\id)\widehat{\tau}\da(x)\\
		&\quad\quad+(\tau\ot\id)(\beta\ot\id)\beta(a)\bullet(\widehat{\tau}\hatot\id)(\db\hatot\id)\db(x)-(\alpha\ot\id)\beta(a)\bullet(\da\hatot\id)\db(x)\\
		&\quad\quad-(\tau\ot\id)(\beta\ot\id)\tau\alpha(a)\bullet(\widehat{\tau}\hatot\id)(\db\hatot\id)\widehat{\tau}\da(x)+(\alpha\ot\id)\tau\alpha(a)
		\bullet(\da\hatot\id)\widehat{\tau}\da(x)\\		
		&\quad=\big((\id\otimes\beta)\beta(a)+(\tau\otimes \id)((\alpha+\beta)\otimes \id)\beta(a)-((\alpha+\beta)\otimes \id)\beta(a)\\
		&\quad\quad-(\tau\otimes \id)(\id\otimes \beta)\beta(a)\big)\bullet(\id \hatot \db)\db(x)-(\id\ot \tau)(\tau\ot \id)\big((\alpha\otimes \id)\alpha(a)\\
		&\quad\quad+(\tau\otimes \id)(\id\otimes \alpha)\beta(a)-(\id\otimes(\alpha+\beta))\alpha(a)-(\tau\otimes \id)(\beta\otimes \id)\alpha(a)\big)\bullet(\id\hatot\widehat{\tau})(\id\hatot\da)\db(x)\\
		&\quad\quad+(\tau\ot \id)(\id\ot \tau)(\tau\ot \id)\big((\alpha\otimes \id)\alpha(a)+(\tau\otimes \id)(\id\otimes \alpha)\beta(a)-(\id\otimes(\alpha+\beta))\alpha(a)\\
		&\quad\quad-(\tau\otimes \id)(\beta\otimes \id)\alpha(a)\big)\bullet(\widehat{\tau}\hatot\id)(\id\hatot\widehat{\tau})(\id\hatot\da)\db(x)\\
		&\quad\quad+\big((\id \otimes \tau)(\beta\otimes \id)\alpha(a)-((\alpha+\beta)\otimes \id)\beta(a)\big)\bullet(\db\hatot \id)\db(x)\\
		&\quad\quad-(\tau\ot \id)\big((\id \otimes \tau)(\beta\otimes \id)\alpha(a)-((\alpha+\beta)\otimes \id)\beta(a)\big)\bullet( \widehat{\tau}\hatot\id)(\db\hatot\id)\db(x)\\
		&\quad\quad+(\id\ot\tau)(\tau\ot\id)\big((\id\otimes\tau)(\alpha\otimes \id)\alpha(a)-(\alpha\otimes \id)\alpha(a)\big)\bullet( \widehat{\tau}\hatot\id)(\id\hatot \widehat{\tau})(\id\hatot \da)\widehat{\tau}\da(x)\\
		&\quad=0.
\end{flalign*}}
Hence Eq. \eqref{co-pL} holds, that is,  $(L, \delta)$ is a \complete pre-Lie coalgebra.

\end{proof}

Finally, we present the coalgebra version of Theorem \ref{t5}.

\begin{theorem}\label{t6}
Let $A$ be a vector space with two linear maps $\alpha$, $\beta: A\rightarrow A\otimes A$, and $(B={\bf k}[t,t^{-1}], \Delta)$ be the \complete right Novikov coalgebra given in Example \mref{Laurent-coproduct}. Then
 $(L, \delta)$ is a \complete pre-Lie coalgebra with $\delta$ defined by Eq. (\ref{co-dipL}) if and only if $(A,\alpha,\beta)$ is a pre-Novikov coalgebra.
\end{theorem}
\begin{proof}
 If $(A,\alpha,\beta)$ is a pre-Novikov coalgebra,  then $(L, \delta)$ is a \complete pre-Lie coalgebra by Theorem \ref{tt6}. Therefore, we only need to prove the ``only if " part. Let $\alpha(a)=\sum_{[a]}
a_{[1]}\otimes a_{[2]}$ and $\beta(a)=\sum_{(a)}a_{(1)}\ot a_{(2)}$ for $a\in A$. Then
$\delta$ is given by
 \begin{eqnarray}
\delta(at^k)=\sum_{i\in \mathbb{Z}}(i+1)(\sum_{(a)}a_{(1)}t^{-i-2}\otimes a_{(2)}t^{k+i}-\sum_{[a]}a_{[2]}t^{k+i}\otimes a_{[1]}t^{-i-2}), \quad  a\in A, k\in \mathbb{Z},\label{pre-Lie-coproduct}
\end{eqnarray}
where $at^k:=a\ot t^k.$
Applying the definition of $\delta$, we obtain
{\small
\begin{flalign*}
&0=((\id\hatot\delta)\delta-(\tau\hatot \id)(\id \hatot \delta)\delta-(\delta\hatot \id)\delta+(\tau\hatot\id)(\delta\hatot \id)\delta)(at^k)\\
&\;\;=\sum_{i,j\in \mathbb{Z}}(i+1)(j+1)\sum_{(a),[a]} (  a_{(1)}t^{-i-2}\ot a_{(21) }t^{-j-2}\ot a_{(22) }t^{i+j+k}-a_{ (1)}t^{-i-2}\ot a_{(2)[2] }t^{i+j+k}\ot a_{(2)[1] }t^{-j-2}   \\
&\quad\quad- a_{[2] }t^{i+k}\ot a_{[1](1) }t^{-j-2}\ot a_{[1](2) }t^{-i-2+j}+ a_{[2] }t^{i+k}\ot a_{[12] }t^{-i-2+j}\ot a_{[11] } t^{-j-2}\\
&\quad\quad- a_{(21) }t^{-j-2}\ot a_{(1) }t^{-i-2}\ot a_{(22) } t^{i+j+k}+ a_{(2)[2] } t^{i+j+k}\ot a_{(1) }t^{-i-2}\ot a_{(2)[1] }t^{-j-2} \\
&\quad\quad+ a_{[1](1) }t^{-j-2}\ot a_{[2]}t^{i+k}\ot a_{[1]( 2)}t^{-i-2+j}- a_{[12] }t^{-i-2+j}\ot a_{ [2]}t^{i+k}\ot a_{[11] } t^{-j-2}\\
&\quad\quad- a_{(11) }t^{-j-2}\ot a_{(12) }t^{-i-2+j}\ot a_{(2) }t^{i+k}+ a_{(1)[2] }t^{-i-2+j}\ot a_{(1)[1] }t^{-j-2}\ot a_{(2) } t^{i+k}\\
&\quad\quad+ a_{[2](1) }t^{-j-2}\ot a_{[2](2) }t^{i+j+k}\ot a_{[1] }t^{-i-2}-a_{[22] }t^{i+j+k}\ot a_{[21] }t^{-j-2}\ot a_{[1] } t^{-i-2}\\
&\quad\quad+ a_{(12) }t^{-i-2+j}\ot a_{(11) }t^{-j-2}\ot a_{(2) }t^{i+k}- a_{(1)[1] }t^{-j-2}\ot a_{(1)[2] }t^{-i-2+j}\ot a_{(2) } t^{i+k}\\
&\quad\quad- a_{[2](2) }t^{i+j+k}\ot a_{[2](1) }t^{-j-2}\ot a_{[1] }t^{-i-2}+ a_{[21] }t^{-j-2}\ot a_{[22] }t^{i+j+k}\ot a_{[1] }t^{-i-2} ).
\end{flalign*}
}

Let $k=2$. Comparing the coefficients of $t^{-1}\otimes t^{-1}\otimes 1$, we obtain
\begin{flalign*}
0=\sum_{[a]} (a_{[2]}\otimes a_{[12]}\otimes a_{[11]}-a_{[12]} \otimes a_{[2]}\otimes a_{[11]})=\tau_{13}\tau_{23}((\id\otimes\tau)(\alpha\otimes \id)\alpha(a)-(\alpha\otimes \id)\alpha(a)),
\end{flalign*}
where $\tau_{13}(a_1\otimes a_2\otimes a_3)=a_3\otimes a_2\otimes a_1$, $\tau_{23}(a_1\otimes a_2\otimes a_3)=a_1\otimes a_3\otimes a_2$, for all $a_1$, $a_2$, $a_3\in A$.
 Comparing the coefficients of $1\otimes t^{-1}\otimes t^{-1}$, we get
\begin{flalign*}
0&=\sum_{(a),[a]} (a_{[1](1)}\otimes a_{[2]}\otimes a_{[1](2)}-a_{(11)} \otimes a_{(12)}\otimes a_{(2)}-a_{(1)[1]}\otimes a_{(1)[2]}\otimes a_{(2)})\\
&=(\id \otimes \tau)(\beta\otimes \id)\alpha(a)-((\alpha+\beta)\otimes \id)\beta(a).
\end{flalign*}
Similarly, let $k=0$. Comparing the coefficients of $t^{-1}\otimes t^{-3}\otimes 1$,
we obtain
\begin{flalign*}
	0&=\sum_{(a),[a]} (a_{[12]}\otimes a_{[2]}\otimes a_{[11]}+a_{(2)[2]}\otimes a_{(1)}\otimes a_{(2)[1]}-a_{[22]}\otimes a_{[21]}\otimes a_{[1]}\\
	&\quad-a_{[2](2)}\otimes a_{[2](1)}\otimes a_{[1]}-a_{(12)} \otimes a_{(11)}\otimes a_{(2)}-a_{(1)[2]}\otimes a_{(1)[1]}\otimes a_{(2)})\\
	&=\tau_{13}((\alpha\otimes \id)\alpha(a)+(\tau\otimes \id)(\id\otimes \alpha)\beta(a)-(\id\otimes(\alpha+\beta))\alpha(a)-(\tau\otimes \id)(\beta\otimes \id)\alpha(a))\\
	&\quad +(\tau\ot \id)((\id \otimes \tau)(\beta\otimes \id)\alpha(a)-((\alpha+\beta)\otimes \id)\beta(a))\\
	&\quad-\tau_{13}\tau_{23}((\id\otimes\tau)(\alpha\otimes \id)\alpha(a)-(\alpha\otimes \id)\alpha(a)).
\end{flalign*}
 Comparing the coefficients of $1\otimes t^{-3}\otimes t^{-1}$, we obtain
\begin{flalign*}
0&=\sum_{(a),[a]} (a_{(1)}\ot a_{(21)}\ot a_{(22)}+a_{[2]}\otimes a_{[1](1)} \otimes a_{[1](2)}-a_{[1](1)}\otimes a_{[2]}\otimes a_{[1](2)}-a_{(21)}\ot a_{(1)}\ot a_{(22)})\\
&=(\id\otimes\beta)\beta(a)+(\tau\otimes \id)((\alpha+\beta)\otimes \id)\beta(a)-((\alpha+\beta)\otimes \id)\beta(a)-(\tau\otimes \id)(\id\otimes \beta)\beta(a)\\
&\quad +(\tau\ot \id)((\id \otimes \tau)(\beta\otimes \id)\alpha(a)-((\alpha+\beta)\otimes \id)\beta(a))\\
&\quad -((\id \otimes \tau)(\beta\otimes \id)\alpha(a)-((\alpha+\beta)\otimes \id)\beta(a)).
\end{flalign*}
Obviously, Eqs. \eqref{cob1}-\eqref{cob4} hold, hence $(A, \alpha ,\beta)$ is a pre-Novikov coalgebra.

This completes the proof.
\end{proof}

\section{Completed pre-Lie bialgebras from pre-Novikov bialgebras and quadratic right Novikov algebras}
In this section, we first give the reason why we choose a quadratic right Novikov algebra instead of a right Novikov dialgebra with a special bilinear form in the finite dimensional case. Then we give a construction of compelted pre-Lie bialgebras from pre-Novikov bialgebras and quadratic $\ZZ$-graded right Novikov algebras and lift the affinization of pre-Novikov algebras to the context of bialgebras.

\subsection{Constructions of finite-dimensional pre-Lie bialgebras from the view of quadratic algebras}\label{3.1}
First, we recall the definition of pre-Novikov bialgebras.
\begin{definition}\cite{LH}
Let $(A, \lhd, \rhd)$ be a pre-Novikov algebra and $(A,\alpha,\beta)$ be a pre-Novikov coalgebra. If they also satisfy the following conditions:
\begin{align}
(\tau\alpha+\beta)(a\circ b)&=((L_{\vartriangleright}+2R_{\vartriangleleft})(a)\otimes \id+ \id\otimes L_{\circ}(a))(\tau\alpha+\beta)(b)\label{lfd1}\\
&\quad+(\id\otimes R_{\circ}(b))(2\tau\alpha+\beta)(a)-(R_{\vartriangleleft}(b)\otimes\id)\tau\alpha (a),\nonumber\\
\tau\alpha(a\circ b-b\circ a)&=((L_{\vartriangleright}+R_{\vartriangleleft})(a)\otimes\id+\id\otimes L_{\circ}(a))\tau\alpha (b)\label{lfd2}\\
&\quad-((L_{\vartriangleright}+R_{\vartriangleleft})(b)\otimes\id+\id\otimes L_{\circ}(b))\tau\alpha (a),\nonumber\\
(\alpha+\beta)(a\vartriangleright b+b\vartriangleleft a)&=(\id\otimes (R_{\vartriangleright}+L_{\vartriangleleft})(b))(2\tau\alpha+\beta)(a)-(L_{\vartriangleleft}(b)\otimes\id)\alpha (a)\label{lfd3}\\
&\quad+((L_{\vartriangleright}+2R_{\vartriangleleft})(a)\otimes\id+\id\otimes(L_{\vartriangleright}+R_{\vartriangleleft})(a))(\alpha+\beta)(b),\nonumber\\
(\alpha+\beta-\tau\alpha-\tau\beta)(b\vartriangleleft a)&=(\id\otimes L_{\vartriangleleft}(b))(\tau\alpha+\beta)(a)-( L_{\vartriangleleft}(b)\otimes \id)(\alpha+\tau\beta)(a)\label{lfd4}\\
&\quad+(\id\otimes R_{\vartriangleleft}(a))(\alpha+\beta)(b)-(R_{\vartriangleleft}(a)\otimes\id)(\tau\alpha+\tau\beta)(b),\nonumber\\
(\id\otimes R_{\circ}(b)-R_{\vartriangleleft}&(b)\otimes\id)(\tau\alpha+\beta)(a)=(\id\otimes R_{\circ}(a)-R_{\vartriangleleft}(a)\otimes\id)(\tau\alpha+\beta)(b),\label{lfd5}\\
\tau\alpha(a\circ b)&=(\id\otimes R_{\circ}(b))\tau\alpha(a)+((L_{\vartriangleright}+R_{\vartriangleleft})(a)\otimes\id)(\tau\alpha+\beta)(b),\label{lfd6}\\
(\id\otimes(R_{\vartriangleright}+L_{\vartriangleleft})(b))\tau\alpha(a)&=((R_{\vartriangleright}+L_{\vartriangleleft})(b)\otimes \id)\alpha(a)+(\id \otimes(L_{\vartriangleright}+R_{\vartriangleleft})(a))(\tau\alpha+\tau\beta)(b)\label{lfd7}\\
&\quad-((L_{\vartriangleright}+R_{\vartriangleleft})(a)\otimes\id)(\alpha+\beta)(b),\nonumber\\
(\alpha+\beta)(b\vartriangleleft a)&=(\id\otimes(R_{\vartriangleright}+L_{\vartriangleleft})(b))(\tau\alpha+\beta)(a)+(R_{\vartriangleleft}(a)\otimes\id)(\alpha+\beta)(b),\;\;a, b\in A,\label{lfd8}
\end{align}
then we call $(A,\vartriangleleft,\vartriangleright,\alpha,\beta)$ a \textbf{pre-Novikov bialgebra}.
\end{definition}

\begin{example}\label{ex1}\cite{LH}
Let $(A=\mathbf{k}e_1\oplus \mathbf{k}e_2,\vartriangleleft,\vartriangleright)$ be a two-dimensional vector space with binary operations
$\vartriangleleft,\vartriangleright$ given by
\begin{flalign*}
&e_i\vartriangleright e_j=0, \;\;i,j\in\{1,2\},\\
e_1\vartriangleleft e_1=e_1,\quad &e_1\vartriangleleft e_2=e_2,\quad e_2\vartriangleleft e_1=e_2,\quad e_2\vartriangleleft e_2 =0.
\end{flalign*}
Then $(A,\vartriangleleft,\vartriangleright)$ is a pre-Novikov algebra. Define linear maps $\alpha,\beta:A\rightarrow A\otimes A$ by
\begin{flalign*}
&\alpha(e_1)=e_2\otimes e_2, \quad \ \ \alpha(e_2)=0,\\
&\beta(e_1)=-e_2\otimes e_2, \quad  \beta(e_2)=0.
\end{flalign*}
Then $(A,\vartriangleleft,\vartriangleright,\alpha,\beta)$ is a pre-Novikov bialgebra.
\end{example}

\begin{definition}\cite{HBG}
	Let $(A,\circ )$ be a Novikov algebra. If there is a skew-symmetric nondegenerate bilinear form $\omega (\cdot,\cdot)$ on $A$ satisfying
	\begin{align}
		\omega (a\circ b,c)-\omega (a\circ c+c\circ a,b)+\omega (c\circ b,a)=0 ,\quad a,b,c\in A,\label{qn}
	\end{align}
	then $(A,\circ ,\omega(\cdot,\cdot))$ is called a \textbf{quasi-Frobenius Novikov algebra}.
\end{definition}

\begin{proposition}\cite[Theorem 2.10]{LH}\label{qN-PN}
	Let $(A,\circ ,\omega(\cdot,\cdot))$ be a quasi-Frobenius Novikov algebra. Then there is a compatible pre-Novikov algebra structure on $A$ given by
	\begin{align}
		\omega (a\vartriangleright b,c)&=\omega(a\circ c+c\circ a,b),\label{t11}\\
		\omega(a\vartriangleleft b,c)&=\omega(a,c\circ b),\quad a,b,c\in A,\label{t12}
	\end{align}
	such that $(A,\circ )$ is the associated Novikov algebra of $(A,\vartriangleleft,\vartriangleright)$. This pre-Novikov algebra is called the \textbf{associated pre-Novikov algebra} of $(A,\circ ,\omega(\cdot,\cdot))$.
\end{proposition}
\begin{definition}
Let $(A,\vartriangleleft,\vartriangleright)$ be a pre-Novikov algebra. If there is a skew-symmetric nondegenerate bilinear form $\omega(\cdot,\cdot)$ on $A$ such that Eqs. (\ref{t11}) and (\ref{t12}) hold, where $a\circ b=a \vartriangleleft b+a\vartriangleright b$ for all $a$, $b\in A$, then $(A,\vartriangleleft,\vartriangleright, \omega(\cdot,\cdot))$ is called a {\bf quadratic pre-Novikov algebra}.
\end{definition}

\begin{definition}\cite{LH}
	Let $(A,\vartriangleleft,\vartriangleright)$, $(A^*,\vartriangleleft_\ast,\vartriangleright_\ast)$ be pre-Novikov algebras and $(A,\circ)$, $(A^\ast,\circ_\ast)$ be their associated Novikov algebras respectively. If a quasi-Frobenius Novikov algebra $(B, \cdot, \omega(\cdot,\cdot))$ satisfies the following conditions:
	\begin{enumerate}
		\item $B$ is the direct sum of $A$ and $A^*$ as vector spaces,
		\item $(A,\vartriangleleft,\vartriangleright)$ and $(A^*,\vartriangleleft_\ast,\vartriangleright_\ast)$ are pre-Novikov subalgebras of $(B,\trianglelefteq,\trianglerighteq)$, which is the associated pre-Novikov algebra of $(B, \cdot, \omega(\cdot,\cdot))$,
		\item the bilinear form $\omega(\cdot,\cdot)$ on $B=A\oplus A^*$ is given by
		\begin{flalign}
			\omega(a+f,b+g)=\langle f,b\rangle -\langle g,a\rangle,\quad  a,b\in A,\;\;f, g\in A^*,\label{man}
		\end{flalign}
	\end{enumerate}
	then  $(B, \cdot, \omega(\cdot,\cdot))$ is called a {\bf double construction of quasi-Frobenius Novikov algebras} associated to $(A,\vartriangleleft,\vartriangleright)$ and $(A^*,\vartriangleleft_\ast,\vartriangleright_\ast)$. Denote it by $(A\oplus A^*,A,A^*,\omega(\cdot,\cdot))$.
\end{definition}
\begin{remark}\label{rmk1}
In this case, $(B,\trianglelefteq,\trianglerighteq,  \omega(\cdot,\cdot))$ is a quadratic pre-Novikov algebra.
\end{remark}
\delete{\begin{definition}
	Let $(A,\vartriangleleft,\vartriangleright)$, $(A^*,\vartriangleleft_\ast,\vartriangleright_\ast)$ be pre-Novikov algebras. If a  pre-Novikov algebra $(B, \trianglelefteq,\trianglerighteq, \omega(\cdot,\cdot))$ satisfies the following conditions:
	\begin{enumerate}
			\item $B$ is the direct sum of $A$ and $A^*$ as vector spaces,
			\item $(A,\vartriangleleft,\vartriangleright)$ and $(A^*,\vartriangleleft_\ast,\vartriangleright_\ast)$ are pre-Novikov subalgebras of $(B,\trianglelefteq,\trianglerighteq)$,
			\item the bilinear form $\omega(\cdot,\cdot)$ on $B=A\oplus A^*$ is given by
			\begin{flalign}
					\omega(a+f,b+g)=\langle f,b\rangle -\langle g,a\rangle,\quad  a,b\in A,\;\;f, g\in A^*,
				\end{flalign}
			and satisfies Eqs. \eqref{t11}-\eqref{t12},
		\end{enumerate}
	then  $(B, \cdot, \omega(\cdot,\cdot))$ is called a {\bf main triple of pre-Novikov algebras} associated to $(A,\vartriangleleft,\vartriangleright)$ and $(A^*,\vartriangleleft_\ast,\vartriangleright_\ast)$. Denote it by $(A\oplus A^*,A,A^*,\omega(\cdot,\cdot))$.
\end{definition}

}

\begin{proposition}\cite[Theorem 3.7]{LH}\label{t4}
	Let $(A,\vartriangleleft,\vartriangleright)$ be a pre-Novikov algebra and $(A,\circ )$ be the associated Novikov algebra of $(A,\vartriangleleft,\vartriangleright)$. Suppose that there is a pre-Novikov algebra $(A^\ast, \vartriangleleft_*, \vartriangleright_*)$ which is induced from a pre-Novikov coalgebra $(A, \alpha, \beta)$, whose associated Novikov algebra is denoted by $(A^\ast, \circ_\ast)$. Then $(A,\vartriangleleft,\vartriangleright,\alpha,\beta)$ is a pre-Novikov bialgebra if and only if there is a double construction of quasi-Frobenius Novikov algebras associated to $(A, \vartriangleleft,$ $\vartriangleright)$ and $(A^\ast, \vartriangleleft_\ast,\vartriangleright_\ast)$.
\delete{the following conditions are equivalent.
	\begin{enumerate}
		\item There is a double construction of quasi-Frobenius Novikov algebras associated to $(A, \vartriangleleft,$ $\vartriangleright)$ and $(A^\ast, \vartriangleleft_\ast,\vartriangleright_\ast)$;
 \item There is a main triple of pre-Novikov algebras associated to $(A, \vartriangleleft,$ $\vartriangleright)$ and $(A^\ast, \vartriangleleft_\ast,\vartriangleright_\ast)$;
		\item $(A,\vartriangleleft,\vartriangleright,\alpha,\beta)$ is a pre-Novikov bialgebra.
	\end{enumerate}}
\end{proposition}
\delete{ \begin{proof}
	It is straightforward that (a) and (b) are equivalent. Moreover, (a) and (c) are equivalent by \cite[Theorem 3.7]{LH}.
\end{proof}}

%A pre-Lie bialgebra is equivalent to a para-K$\ddot{\rm{a}}$hler Lie algebra  which is a Lie algebras with nondegenerate forms (see \cite{Bai1}). Similarly,  a pre-Novikov bialgebra is equivalent to a double construction of quasi-Frobenius Novikov algebras (see \cite{LH}).   However, the tensor product of a  Novikov algebra and a  right Novikov dialgebra generally fails to form a Lie algebra. Therefore, we  only consider constructing  a completed pre-Lie bialgebra from a finite-dimensional pre-Novikov bialgebra via a quadratic $\ZZ$-graded right Novikov algebra.

Next, we recall some results about pre-Lie bialgebras.

\begin{definition}\cite{Bai1, LZB} \label{pre-lia3}
Let $(L, \circ)$ be a pre-Lie algebra and $(L,
\delta)$ be a pre-Lie coalgebra. If the following
compatibility conditions are satisfied:
\begin{align}
&(\delta-\tau\delta)(a\circ b)-(L_{\circ}(a)\otimes \id)(\delta-\tau\delta)(b)-(\id\otimes L_{\circ}(a))(\delta-\tau\delta)(b)\label{eq:pre-Lbi1}\\
&\quad-(\id\otimes R_{\circ}(b))\delta(a)+(R_{\circ}(b)\otimes \id)\tau\delta(a)=0, \nonumber\\
&\delta(a\circ b-b\circ a)-(\id\otimes (R_{\circ}(b)-L_{\circ}(b)))\delta(a)-(\id \otimes(L_{\circ}(a)-R_{\circ}(a)))\delta(b)\label{eq:pre-Lbi2}\\
&\quad-(L_{\circ}(a)\otimes\id)\delta(b)+(L_{\circ}(b)\otimes\id)\delta(a)=0,\quad     a, b\in L,\nonumber
\end{align}
then we call $(L, \circ,\delta)$ a {\bf pre-Lie bialgebra}.
\end{definition}

Recall that
a Lie algebra $(\mathcal{G},[\cdot, \cdot])$ is called a  \textbf{symplectic Lie algebra} (or \textbf{quasi-Frobenius Lie algebra}) if there is a nondegenerate skew-symmetric bilinear form  $\omega_p(\cdot,\cdot)$ (called a 2-cocycle)  on  $ \mathcal{G} $, that is,
	\begin{align}
		\omega_p ([x,y],z)+\omega_p ([y,z],x)+\omega_p ([z,x],y)=0 ,\quad x,y,z\in \mathcal{G}.
	\end{align}
	We denote it by $(\mathcal{G}, [\cdot, \cdot],\omega_p(\cdot, \cdot))$.
	
	A symplectic Lie algebra $(\mathcal{G}, [\cdot, \cdot],\omega_p(\cdot, \cdot))$ is called a {\bf para-K$\ddot{\texttt{a}}$hler Lie algebra} if $\mathcal{G}=\mathcal{G}_1\oplus\mathcal{G}_2$ is the direct sum of the underlying vector spaces of two Lie subalgebras in which $\omega_p(\mathcal{G}_i,\mathcal{G}_i)=0$ for $i=1,2$. Denote it by $(\mathcal{G}_1\bowtie\mathcal{G}_2,\mathcal{G}_1,\mathcal{G}_2,\omega_p(\cdot, \cdot))$.

\begin{proposition}\cite{Bai1}\label{qL-PL}
	Let	$(\mathcal{G},[\cdot, \cdot],\omega_p(\cdot, \cdot))$ be a symplectic Lie algebra. Then there is  a compatible pre-Lie algebra structure $\circ$ on $\mathcal{G}$ given by
	\begin{align}\label{t13}
		\omega_p(x\circ y,z)=-\omega_p(y,[x,z]),\quad x,y,z\in \mathcal{G}.
	\end{align}
\end{proposition}

\begin{definition}\cite{LZB}
Let $(L, \circ)$ be a pre-Lie algebra. If  there is a nondegenerate skew-symmetric bilinear form $\omega_p(\cdot,\cdot)$ on $L$ such that Eq. (\ref{t13}) holds, where $[x,z]=x\circ z-z\circ x$ for all $x$, $z\in L$, then $(L, \circ, \omega_p(\cdot,\cdot))$ is called a {\bf quadratic pre-Lie algebra}.
\end{definition}

\begin{remark}\label{rmk2}
By Proposition \ref{qL-PL}, for a para-K$\ddot{\texttt{a}}$hler Lie algebra $(\mathcal{G}, [\cdot, \cdot],\omega_p(\cdot, \cdot))$, there is a  quadratic pre-Lie algebra $(\mathcal{G}, \circ,\omega_p(\cdot, \cdot))$.
\end{remark}
\delete{
\begin{definition}
Let	$(A,\circ)$ be a  pre-Lie algebra. Suppose that $(A^\ast, \circ^\ast)$ is a pre-Lie algebra which induced from a pre-Lie coalgebra $(A,\delta)$. If a pre-Lie  algebra $(B, \cdot, \omega_p(\cdot,\cdot))$ satisfies the following conditions: ( for any $x,y\in A,x^*,y^*\in A^*)$
	\begin{enumerate}
			\item $B$ is the direct sum of $A$ and $A^*$ as vector spaces,
			\item 	$(A,\circ)$ and $(A^\ast, \circ^\ast)$ are pre-Novikov subalgebras of $(B, \cdot)$,
			\item the bilinear form $\omega_p(\cdot,\cdot)$ on $B=A\oplus A^*$ is given by
			\begin{flalign}\label{t14}
					\omega_p(x+x^*,y+y^*)=\langle x^*,y\rangle -\langle y^*,x\rangle,
				\end{flalign}
		and satisfies Eq. \eqref{t13},
%\begin{flalign}
%		 \omega_p((a+f)\cdot (b+g),c+h)=-\omega_p(b+g, (a+f)\cdot(c+h)-(c+h)\cdot(a+f)),
%		 	\end{flalign}
\end{enumerate}	
then  $(B, \cdot, \omega_p(\cdot,\cdot))$ is called a {\bf main triple of pre-Lie algebras} associated to 	$(A,\circ)$ and $(A^\ast, \circ^\ast)$. Denote it by $(A\oplus A^*,A,A^*,\omega_p(\cdot,\cdot))$.
\end{definition}}

\begin{proposition}\cite[Proposition 4.2]{Bai1}\label{t41}
	Let	$(A,\circ)$ be a  pre-Lie algebra. Suppose that $(A^\ast, \circ^\ast)$ is a pre-Lie algebra which is induced from a pre-Lie coalgebra $(A,\delta)$.  Then $(A,\circ ,\delta)$ is a pre-Lie bialgebra if and only if $(\mathcal{G}(A)\bowtie \mathcal{G}(A^\ast),\mathcal{G}(A),\mathcal{G}(A^*),\omega_p(\cdot,\cdot))$ is a para-K$\ddot{a}$hler Lie algebra, where $\omega_p(\cdot,\cdot)$ is given by
\begin{eqnarray*}
\omega_p(a+f,b+g)=\langle f,b\rangle -\langle g,a\rangle,\quad  a,b\in A,\;\;f, g\in A^*.
\end{eqnarray*}
\end{proposition}
 \delete{\begin{proof}
	It is straightforward that (a) and (b) are equivalent. Moreover, (a) and (c) are equivalent by \cite[Proposition 4.2]{Bai1}.
\end{proof}}

\delete{\begin{proposition}\label{qN-PaL}\cite[proposition 4.9]{HBG}
	Let $(A,\circ ,\omega(\cdot,\cdot))$ be a quasi-Frobenius Novikov algebra and $(B=\oplus_{i\in \ZZ}B_i, \diamond, (\cdot, \cdot))$ be a
	quadratic $\ZZ$-graded right Novikov algebra. Define a binary operation $[\cdot, \cdot]$ on $A\otimes B$ by
	\begin{align}
		[a\ot x,b\ot y]=a\circ b \ot x\diamond y-b \circ a\otimes y\diamond x,\quad  a,b\in A,x,y\in B.
	\end{align}
	Then $(\mathcal{G}:=A\otimes B, [\cdot, \cdot],\omega_p(\cdot,\cdot))$ is a symplectic Lie algebra, where $\omega_p(\cdot,\cdot)$ is given by
	\begin{align}\label{om-p}
		\omega_p(a\ot x,b\ot y)=\omega(a,b)(x,y),\quad a,b\in A,x,y\in B.
	\end{align}
\end{proposition}}

%\begin{proof}
%	It is obvious that $(\mathcal{G},[\cdot, \cdot])$ is a Lie algebra by \cite[Theorem 2.9]{HBG}.  Therefore we only need to show $\omega_p$ is a 2-cocycle on $\mathcal{G}$. By Eq. \eqref{qn}, for any $a,b,c\in A$, $x,y \in B$, we have
%	\begin{align*}
%		&\omega_p([a\otimes x,b\otimes y],c\otimes z)+\omega_p([b\otimes y, c\otimes z],a\otimes x )+\omega_p([c\otimes z,a\otimes x ],b\otimes y)\\
%		&\quad  =\omega(a\circ b,c)(x\diamond y, z)-\omega(b\circ a,c)(y \diamond  x,z)\\
%		&\quad\quad+\omega(b \circ c,a)(y\diamond  z,x)-\omega(c\circ b,a)(z\diamond  y,x)\\
%		&\quad\quad+\omega(c\circ a,b)(z \diamond x,y)-\omega(a\circ c,b)(x\diamond  z,y)\\
%		&\quad =\omega(a\circ b,c)((x \diamond y , z)+(y\diamond z,x)+(z\diamond y,x))+\omega(b\circ a,c)(-(y\diamond x,z)+(y\diamond z,x))\\
%		&\quad\quad+\omega(a\circ c,b)(-(y\diamond z,x)-(z\diamond y,x)-(x\diamond z,y))+\omega(c\circ a,b)(-(z\diamond y,x)+(z\diamond x,y))\\
%		&\quad=0.
%	\end{align*}
%\end{proof}
Note that our goal is to lift the conclusion of Theorem \ref{di-pre-L} to a bialgebraic version. \delete{Let $(A, \triangleleft, \triangleright, \alpha, \beta)$ be a pre-Noivkov bialgebra and $(B, \dashv, \vdash)$ be a right Novikov dialgebra. By Proposition \ref{t41}, a pre-Lie bialgebra $(L=A\otimes B, \circ, \delta)$ is equivalent to a para-K$\ddot{a}$hler Lie algebra $(\mathcal{G}(L)\bowtie \mathcal{G}(L^\ast),\mathcal{G}(L),\mathcal{G}(L^*),\omega_p(\cdot,\cdot))$$=$$(\mathcal{G}(A\otimes B)\bowtie \mathcal{G}((A^\ast\otimes B^\ast),\mathcal{G}(A\otimes B),\mathcal{G}(A^*\otimes B^\ast),\omega_p(\cdot,\cdot))$. By Proposition \ref{t1}, $(A, \triangleleft, \triangleright, \alpha, \beta)$ is equivalent to a double construction $(A\oplus A^*,A,A^*,\omega(\cdot,\cdot))$ of quasi-Frobenius Novikov algebras. Therefore, it is natural to require that}
Motivated by the method used in \cite{HBG}, in the finite-dimensional case, it is natural to consider whether the tensor product $L=A\otimes B$ of a pre-Novikov bialgebra $(A,\vartriangleleft, \vartriangleright, \alpha, \beta)$ and  a  right Novikov dialgebra $(B,\dashv,\vdash)$ with a special nondegenerate bilinear form $(\cdot,\cdot)$ can be endowed with a natural structure of pre-Lie bialgebras. By Propositions \ref{t4} and \ref{t41}, it is equivalent to consider whether the tensor product of $(A\oplus A^*,A,A^*,\omega(\cdot,\cdot)) $
and a right Novikov dialgebra with a special nondegenerate bilinear form can be endowed with a natural para-K$\ddot{\texttt{a}}$hler Lie algebra structure. By Remarks \ref{rmk1} and \ref{rmk2}, we need to consider whether there is a natural quadratic pre-Lie algebra structure on the tensor product of a quadratic pre-Novikov algebra and a right Novikov dialgebra with a special nondegenerate bilinear form. Note that there is a natural bilinear form on the induced pre-Lie algebra which is the product of the bilinear form on the pre-Novikov algebra and the bilinear form on the right Novikov dialgebra. Therefore, we present the following proposition.

\begin{proposition}\label{pro-q}
Let $(A, \vartriangleleft, \vartriangleright, \omega(\cdot,\cdot))$ be a quadratic pre-Novikov algebra, $(B, \dashv, \vdash)$ be a right Novikov dialgebra with a nondegenerate bilinear form $(\cdot,\cdot)$, and $(L:=A\otimes B, \circ)$ be the induced pre-Lie algebra from
$(A,\vartriangleleft, \vartriangleright)$ and $(B, \dashv, \vdash)$. Then $(L, \circ, \omega_p(\cdot,\cdot))$ is a quadratic pre-Lie algebra with $\omega_p(\cdot,\cdot)$ defined by
\begin{eqnarray}\label{def-1}
\omega_p(a\otimes x,b\otimes y):=\omega(a,b)(x, y),\;\;\;a, b\in A,\;\;x, y\in B,
\end{eqnarray}
if and only if $(\cdot,\cdot)$ is symmetric and the following equality holds for all $a$, $b$, $c\in A$ and $x$, $y$, $z\in B$:
\begin{eqnarray}
&&\omega(b\circ c,a)(-(y,x\vdash z)-(y,z\dashv x)+(y,z\vdash x)+(y,x\dashv z))\\
	&&\quad\;\;+\omega(c\circ b,a)(-(y,z\dashv x)+(y,z\vdash x))\nonumber\\
	&&\quad\;\;+\omega(a\circ c,b)((x\vdash y,z)-(y,x\vdash z))\nonumber\\
	&&\quad\;\;+\omega(c\circ a,b)((x\vdash y,z)+(y\dashv x,z)+(y,z\dashv x))=0.\nonumber
\end{eqnarray}
.
\end{proposition}
\begin{proof}
Since $\omega(\cdot,\cdot)$ is skew-symmetric, $\omega_p(\cdot,\cdot)$ is skew-symmetric if and only if $(\cdot,\cdot)$ is symmetric.
Note that \begin{align*}
	&\omega_p((a\ot x)\circ (b\ot y),c\ot z)+\omega_p(b\ot y,[a\ot x,c\ot z ])\\
	&\;\; =\omega(a\vartriangleright b,c)(x\vdash y,z)-\omega(b\vartriangleleft a,c)(y\dashv x,z)\\
	&\quad\;\; +\omega(b,a\vartriangleright c)(y,x\vdash z)-\omega(b,c\vartriangleleft a)(y,z\dashv x)\\
	&\quad\;\;-\omega(b,c\vartriangleright a)(y,z\vdash x)+\omega(b,a\vartriangleleft c)(y,x\dashv z)\\
	&\;\;=\omega(b\circ c,a)(-(y,x\vdash z)-(y,z\dashv x)+(y,z\vdash x)+(y,x\dashv z))\\
	&\quad\;\;+\omega(c\circ b,a)(-(y,z\dashv x)+(y,z\vdash x))\\
	&\quad\;\;+\omega(a\circ c,b)((x\vdash y,z)-(y,x\vdash z))\\
	&\quad\;\;+\omega(c\circ a,b)((x\vdash y,z)+(y\dashv x,z)+(y,z\dashv x)).
\end{align*}
Then this conclusion holds.
\end{proof}

\begin{definition}\cite{HBG}
Let $(B, \diamond)$ be a right Novikov algebra. If there is a nondegenerate bilinear form $(\cdot,\cdot)$ on $B$ satisfying
\begin{eqnarray}\label{Nov-qua}
(x\diamond y, z)=-(x, y\diamond z+z\diamond y),\;\;\;x, y, z\in B,
\end{eqnarray}
then $(B, \diamond, (\cdot,\cdot))$ is called a {\bf quadratic right Novikov algebra}.
\end{definition}
\begin{remark}
By Proposition \ref{pro-q}, since $\omega(b\circ c,a)$, $\omega(c\circ b,a)$, $\omega(a\circ c,b)$, and $\omega(c\circ a,b)$ cannot be expressed in terms of each other in general, comparing the coefficients of $\omega(b\circ c,a)$, $\omega(c\circ b,a) $, $\omega(a\circ c,b)$ and $\omega(c\circ a,b)$, we obtain the following equalities:
\begin{align*}
	&-(y,x\vdash z)-(y,z\dashv x)+(y,z\vdash x)+(y,x\dashv z)=0,\\
	&-(y,z\dashv x)+(y,z\vdash x)=0,\\
	&(x\vdash y,z)-(y,x\vdash z)=0,\\
	&(x\vdash y,z)+(y\dashv x,z)+(y,z\dashv x)=0.
\end{align*}
By the nondegenerate property of $(\cdot,\cdot)$, the above equalities imply that
$(B,\dashv  ,\vdash ,(\cdot,\cdot))$
is a quadratic right Novikov algebra. \delete{specifically,
$\dashv =\vdash $
and Eq. \eqref{Rbilinear1} holds.}

 On the other  hand,  let $(B,\diamond,(\cdot,\cdot))$ be a quadratic right Novikov algebra. It is straightforward to verify that there is a para-K$\ddot{\texttt{a}}$hler Lie algebra $(\mathcal{G}(L:=A\ot B)\bowtie\mathcal{G}(L^*),\mathcal{G}(L),\mathcal{G}(L^*),\omega_p(\cdot, \cdot))$ which is isomorphic to the para-K$\ddot{\texttt{a}}$hler Lie algebra $(\mathcal{G}(L:=A\ot B)\bowtie\mathcal{G}(A^*\otimes B),\mathcal{G}(A\otimes B),\mathcal{G}(A^*\otimes B),\omega_p(\cdot, \cdot))$ obtained from $(A\oplus A^*,A,A^*,\omega(\cdot,\cdot)) $ and  $(B,\diamond,(\cdot,\cdot))$ by Proposition \ref{pro-q}.
 Hence, we only consider to construct  pre-Lie bialgebras by  pre-Novikov bialgebras and  quadratic right Novikov algebras.
\end{remark}

\delete{
Under the same assumption as in Theorem~\ref{t4}. Let $(A\oplus A^*,A,A^*,\omega(\cdot,\cdot)) $ be the double construction of quasi-Frobenius Novikov algebras associated with
$(A, \vartriangleleft, \vartriangleright)$ and $(A^\ast, \vartriangleleft\ast, \vartriangleright\ast)$.
Our objective is to elevate the conclusion of Theorem \ref{t5} to a bialgebraic version. Naturally, we must consider whether the tensor product of a pre-Novikov bialgebra and  a  right Novikov dialgebra $(B,\dashv,\vdash)$ with a special bilinear form $(\cdot,\cdot)$ can possess the structure of a pre-Lie bialgebra.

By Theorems \ref{t4} and \ref{t41}, this assertion is equivalent to the following: the tensor product of $(A\oplus A^*,A,A^*,\omega(\cdot,\cdot)) $
and the aforementioned right Novikov dialgebra (with the specified bilinear form) possesses a para-K?hler Lie algebra structure.
Assuming the conclusion holds, it requires determining the explicit form of the bilinear form on the right Novikov dialgebra.

By Theorem \ref{di-pre-L}, we have known that there is a pre-Lie algebra structure $\circ $ defined by Eq. \eqref{def-di-pL} on the tensor product of a pre-Novikov algebra and a right Novikov dialgebra.  Therefore by Theorem \ref{qN-PN} and \ref{qL-PL}, we have
\begin{align*}
	&0=\omega_p((a\ot x)\circ (b\ot y),c\ot z)+\omega_p(b\ot y,[a\ot x,c\ot z ])\\
	&\;\; =\omega(a\vartriangleright b,c)(x\vdash y,z)-\omega(b\vartriangleleft a,c)(y\dashv x,z)\\
	&\quad\;\; +\omega(b,a\vartriangleright c)(y,x\vdash z)-\omega(b,c\vartriangleleft a)(y,z\dashv x)\\
	&\quad\;\;-\omega(b,c\vartriangleright a)(y,z\vdash x)+\omega(b,a\vartriangleleft c)(y,x\dashv z)\\
	&\;\;=\omega(b\circ c,a)(-(y,x\vdash z)-(y,z\dashv x)+(y,z\vdash x)+(y,x\dashv z))\\
	&\quad\;\;+\omega(c\circ b,a)(-(y,z\dashv x)+(y,z\vdash x))\\
	&\quad\;\;+\omega(a\circ c,b)((x\vdash y,z)-(y,x\vdash z))\\
	&\quad\;\;+\omega(c\circ a,b)((x\vdash y,z)+(y\dashv x,z)+(y,z\dashv x)).
\end{align*}}

\subsection{A general construction of completed pre-Lie bialgebras from pre-Novikov bialgebras and quadratic $\mathbb{Z}$-graded right Novikov algebras}

Based on the discussion in Section \ref{3.1}, for lifting the conclusion of Theorem \ref{di-pre-L} to a bialgebraic version, we need to replace right Novikov dialgebras by right Novikov algebras.

First, we recall some basic facts about quadratic $\mathbb{Z}$-graded right Novikov algebras.

\begin{definition}\mlabel{def:quad}\cite{HBG}
Let $(B=\oplus_{i\in \ZZ}B_i, \diamond)$ be a $\ZZ$-graded right Novikov algebra. A bilinear form
$(\cdot,\cdot)$  on $B$ is called {\bf graded} if there exists some $m\in \ZZ$ such that
\begin{eqnarray*}
(B_i,B_j)=0, \;\;\; \text{for any $i$, $j\in \mathbb{Z}$ satisfying $i+j+m\neq 0$.}
\end{eqnarray*}
\delete{The bilinear form is called {\bf invariant} if it satisfies
\begin{eqnarray}\label{Rbilinear1}
(a\diamond b,c)=-(a, b\diamond c+c\diamond b),\quad a,b,c\in B.
\end{eqnarray}}
A {\bf quadratic  $\ZZ$-graded right Novikov algebra}, denoted by $(B=\oplus_{i\in \ZZ}B_i, \diamond,
(\cdot,\cdot))$, is a $\ZZ$-graded right Novikov algebra $(B,\diamond)$ together
with a symmetric  nondegenerate graded bilinear form
$(\cdot,\cdot)$ satisfying Eq. (\ref{Nov-qua}).  In particular, when $B=B_0$, then it is just a quadratic right Novikov algebra.
\end{definition}
For a quadratic $\ZZ$-graded right Novikov algebra $(B=\oplus_{i\in \ZZ}B_i, \diamond, (\cdot, \cdot))$, the nondegenerate symmetric bilinear form $(\cdot,\cdot)$ induces multilinear forms $(\cdot,\cdot)_k, k\geq 2$, by
\vspace{-.2cm}
{\small
\begin{equation} \label{eq:pairb}
(\cdot,\cdot)_k: (\underbrace{B\hatot \cdots \hatot
B}_{k\text{-fold}}) \ot (\underbrace{B \ot \cdots \ot
B}_{k\text{-fold}}) \to \bfk,
    \Big(\hspace{-.3cm}\sum_{i_1,\cdots,i_k,\alpha} a_{1i_1\alpha}\ot \cdots \ot a_{ki_k\alpha}, b_1\ot \cdots \ot b_k\Big)_k\coloneqq \hspace{-.4cm}\sum_{i_1,\cdots,i_k,\alpha} \prod_{\ell=1}^k(a_{\ell i_\ell\alpha}, b_\ell)
\end{equation}
}
with the notation in Eq.~\eqref{eq:ssum} and homogeneous elements $b_i\in B$.
Further the forms are {\bf left nondegenerate} in the sense that if
$$\Big(\sum_{i_1,\cdots, i_k,\alpha} a_{1i_1\alpha}\ot \cdots \ot a_{ki_k\alpha}, b_1\ot \cdots \ot b_k\Big)_k=\Big(\sum_{j_1,\cdots, j_k,\beta} b_{1j_1\beta}\ot \cdots \ot b_{kj_k\beta}, b_1\ot \cdots \ot b_k\Big)_k
$$
for all homogeneous elements $b_1, \ldots,b_k\in B$, then
$$\sum_{i_1,\cdots, i_k,\alpha} a_{1i_1\alpha}\ot \cdots \ot a_{ki_k\alpha}=\sum_{j_1,\cdots, j_k,\beta} b_{1j_1\beta}\ot \cdots \ot b_{kj_k\beta}.
$$

\begin{example}\label{Laurent-Bilinear}\cite{HBG}
Let $(B=\oplus_{i\in \mathbb{Z}}{\bf k}t^i, \diamond)$ be the $\mathbb{Z}$-graded right Novikov algebra given in Example \ref{rN-ex}.
Define a bilinear form $(\cdot,\cdot)$ on $B$ by
\begin{eqnarray}\label{Laurent-Bilinear-1}
(t^i,t^j)=\delta_{i+j+1,0}, ~~~~~i, j\in \mathbb{Z}.
\end{eqnarray}
Then $(B={\bf k}[t,t^{-1}], \diamond, (\cdot, \cdot))$
is a quadratic $\mathbb{Z}$-graded right Novikov algebra.
\end{example}

\begin{example}\label{quadratic-ex}\cite{HBG}
	Let $(B, \diamond)$ be  a $2$-dimensional right Novikov algebra  with a basis $\{x,y\}$ whose multiplication is given by
	\begin{eqnarray*}
		x\diamond x=0,~~x\diamond y=-2x,~~y\diamond x=x,~~y\diamond y=y.
	\end{eqnarray*}
	Define a bilinear form $( \cdot ,\cdot )$ on $B$ by
	%\vspb
	%\begin{eqnarray*}
	$(x,x)=(y,y)=0,~~(x,y)=(y,x)=1.$
	%\vspb
	%\end{eqnarray*}
 Then $(B, \diamond, (\cdot,\cdot ))$ is a quadratic right Novikov algebra.
\end{example}

\begin{lemma}\mlabel{lem:cop}\cite{HBG}
Let $(B=\oplus_{i\in \ZZ}B_i, \diamond, (\cdot, \cdot))$ be a quadratic $\ZZ$-graded right Novikov algebra. Let $\Delta:B\rightarrow B\hatot B$ be the dual of $\diamond$ under the left nondegenerate bilinear form in Eq.~\eqref{eq:pairb}, that is,
\begin{equation}\label{eq:coproduct-self}
(\Delta(x), y\otimes z)=(x, y\diamond z),\quad x,y,z\in B.
\end{equation}
Then $(B,\Delta)$ is a \complete right Novikov coalgebra.
\end{lemma}

%\begin{definition}\cite{Bai1,ZY}
%Let $(L, \circ)$ be a pre-Lie algebra and $(L,
%\delta)$ be a pre-Lie coalgebra. If the following
%compatibility conditions are satisfied:
%\begin{align*}
%&(\delta-\tau\delta)(a\circ b)-(L_{\circ}(a)\ot \id)(\delta-\tau\delta)(b)-(\id\ot L_{\circ}(a))(\delta-\tau\delta)(b)\\
%&\quad-(\id\ot R_{\circ}(b))\delta(a)+(R_{\circ}(b)\ot \id)\tau\delta(a)=0, \nonumber\\
%&\delta(a\circ b-b\circ a)-(\id\ot (R_{\circ}(b)-L_{\circ}(b)))\delta(a)-(\id \ot(L_{\circ}(a)-R_{\circ}(a)))\delta(b)\\
%&\quad-(L_{\circ}(a)\ot\id)\delta(b)+(L_{\circ}(b)\ot\id)\delta(a)=0,\quad     a, b\in L,\nonumber
%\end{align*}
%then we call $(L, \circ,\delta)$ a {\bf pre-Lie bialgebra}.
%\end{definition}

\begin{definition}\label{pre-lia3}
Let $(L, \circ)$ be a pre-Lie algebra and $(L,
\delta)$ be a \complete pre-Lie coalgebra. If the following
compatibility conditions are satisfied:
\begin{align}
&(\delta-\widehat{\tau}\delta)(a\circ b)-(L_{\circ}(a)\hatot \id)(\delta-\widehat{\tau}\delta)(b)-(\id\hatot L_{\circ}(a))(\delta-\widehat{\tau}\delta)(b)\label{eq:pre-Lbi1}\\
&\quad-(\id\hatot R_{\circ}(b))\delta(a)+(R_{\circ}(b)\hatot \id)\widehat{\tau}\delta(a)=0, \nonumber\\
&\delta(a\circ b-b\circ a)-(\id\hatot (R_{\circ}(b)-L_{\circ}(b)))\delta(a)-(\id \hatot(L_{\circ}(a)-R_{\circ}(a)))\delta(b)\label{eq:pre-Lbi2}\\
&\quad-(L_{\circ}(a)\hatot\id)\delta(b)+(L_{\circ}(b)\hatot\id)\delta(a)=0,\quad     a, b\in L,\nonumber
\end{align}
then we call $(L, \circ,\delta)$ a {\bf \complete pre-Lie bialgebra}. If $L$  is finite-dimensional,  $(L, \circ,\delta)$ is just the usual pre-Lie bialgebra.
\end{definition}

Next, we present our main result.

\begin{theorem}\label{thm-bi}
 Let $(A,\vartriangleleft,\vartriangleright,\alpha,\beta)$ be a pre-Novikov bialgebra and $(B=\oplus_{i\in \ZZ}B_i, \diamond, (\cdot, \cdot))$ be a
quadratic $\ZZ$-graded right Novikov algebra. Let
$(L:=A\otimes B,\circ)$ be the induced pre-Lie algebra from $(A,
\vartriangleleft,\vartriangleright)$ and $(B, \diamond)$, $\Delta:B\rightarrow
B\hatot B$ be the linear map defined by Eq. \eqref{eq:coproduct-self}, and $\delta:L\rightarrow L\widehat{ \otimes }L$ be the linear map defined in Eq. \eqref{co-dipL}. Then $(L, \circ, \delta)$ is a \complete pre-Lie bialgebra. Furthermore, if $(B, \diamond, (\cdot, \cdot))=({\bf k}[t,t^{-1}], \diamond, (\cdot, \cdot))$ is the quadratic $\mathbb{Z}$-graded right Novikov algebra given in Example \ref{Laurent-Bilinear}, then $(L, \circ, \delta)$ is a \complete pre-Lie bialgebra if and only if $(A,\vartriangleleft,\vartriangleright,\alpha,\beta)$ is a pre-Novikov bialgebra.
\end{theorem}

\begin{proof}
By Lemma \mref{lem:cop}, $(B,\Delta)$ is a \complete right Novikov coalgebra. Then by
Theorem \ref{t6}, $(L, \delta)$
is a \complete pre-Lie coalgebra.
Let $a\otimes x, b\otimes y\in A\otimes B$ with $x\in B_i, y\in B_j$. We obtain
\begin{flalign*}
&(\delta-\widehat{\tau}\delta)((a\otimes x) \circ (b\otimes y))\\
	&\quad=\beta(a\vartriangleright b)\bullet\Delta(x\diamond y)-\widehat{\tau}\alpha(a\vartriangleright b)\bullet \widehat{\tau}\Delta(x\diamond y)
-\widehat{\tau}\beta(a\vartriangleright b)\bullet \widehat{\tau}\Delta(x\diamond  y)\\
&\quad\quad+\alpha(a\vartriangleright b)\bullet\Delta(x\diamond y)
-\beta(b\vartriangleleft a)\bullet\Delta(y\diamond x)+\widehat{\tau}\alpha(b\vartriangleleft a)\bullet\widehat{\tau}\Delta(y\diamond x)
\\
&\quad\quad+\widehat{\tau}\beta(b\vartriangleleft a)\bullet\widehat{\tau}\Delta(y\diamond x)-\alpha(b\vartriangleleft a)\bullet\Delta(y\diamond x),
\end{flalign*}
and
\small{
\begin{flalign*}
&(L_{\circ}(a\ot x)\hatot \id)(\delta-\widehat{\tau}\delta)(b\otimes y)\\
&\quad=\sum_{(b)}\sum_{i,j,\gamma}(a\vartriangleright b_{(1)}\ot b_{(2)})\bullet(x\diamond y_{1,i,\gamma}\ot y_{2,j,\gamma})-\sum_{(b)}\sum_{i,j,\gamma}(b_{(1)}\vartriangleleft a\ot b_{(2)})\bullet(y_{1,i,\gamma}\diamond x\ot y_{2,j,\gamma})\\
&\quad\quad-\sum_{[b]}\sum_{i,j,\gamma}(a\vartriangleright b_{[2]}\ot b_{[1]})\bullet(x \diamond y_{2,j,\gamma}\ot y_{1,i,\gamma})+\sum_{[b]}\sum_{i,j,\gamma}(b_{[2]}\vartriangleleft a\ot b_{[1]})\bullet(y_{2,j,\gamma}\diamond x\ot y_{1,i,\gamma})\\
&\quad\quad-\sum_{(b)}\sum_{i,j,\gamma}(a\vartriangleright b_{(2)}\ot b_{(1)})\bullet(x\diamond y_{2,j,\gamma}\ot y_{1,i,\gamma})+\sum_{(b)}\sum_{i,j,\gamma}(b_{(2)}\vartriangleleft a\ot b_{(1)})\bullet(y_{2,j,\gamma}\diamond x\ot y_{1,i,\gamma})\\
&\quad\quad+\sum_{[b]}\sum_{i,j,\gamma}(a\vartriangleright b_{[1]}\ot b_{[2]})\bullet(x\diamond y_{1,i,\gamma}\ot y_{2,j,\gamma})-\sum_{[b]}\sum_{i,j,\gamma}(b_{[1]}\vartriangleleft a\ot b_{[2]})\bullet(y_{1,i,\gamma}\diamond x\ot y_{2,j,\gamma}).
\end{flalign*}}
Let $e\in B_{k}$ and $f\in B_{l}$. Note that
{\small
\begin{eqnarray*}
&&(\widehat{\tau} \Delta(x\diamond y), e\otimes f)=\Big(\sum_{i,j,\gamma} (x\diamond y)_{2j\gamma}\hatot  (x\diamond y)_{1i\gamma}, e\otimes f\Big)=(x\diamond y, f\diamond e)=(y, x\diamond (f\diamond e)),\\
&&\Big(\sum_{i,j,\gamma} (x\diamond y_{1i\gamma}\otimes y_{2j\gamma}), e\otimes f\Big)=\sum_{i,j,\gamma} (y_{1i\gamma}\otimes y_{2j\gamma}, ( x\diamond e)\otimes f)=(y, (x\diamond e)\diamond f)\\
&&=\Big(y, (x\diamond f)\diamond e+x\diamond (e\diamond f)-x\diamond (f\diamond e)\Big).
\end{eqnarray*}
}
By the nondegeneracy of $(\cdot,\cdot)$, we have
\begin{equation}
\label{bialg-eq1}\sum_{i,j,\gamma} (x\diamond y_{1i\gamma}\otimes y_{2j\gamma})=-\widehat{\tau} \Delta(x\diamond y)+\sum_{i,j,\gamma} g_{i\gamma}\otimes h_{j\gamma},
\end{equation}
where $\sum_{i,j,\gamma} g_{i\gamma}\otimes h_{j\gamma}\in B\hatot  B$ is chosen such that $(\sum_{i,j,\gamma} g_{i\gamma}\otimes h_{j\gamma}, e\otimes f)=(y, (x\diamond f)\diamond e+x\diamond (e\diamond f))$.
Similarly, we obtain
{\small
\begin{eqnarray}
\mlabel{bialg-eq2}&&\sum_{i,j,\gamma} (y_{1i\gamma}\otimes x\diamond y_{2j\gamma})=\Delta(x\diamond y),\;\;\sum_{i,j,\gamma} (y_{1i\gamma}\otimes y_{2j\gamma}\diamond x)=-\Delta(x\diamond y)-\sum_{i,j,\gamma} k_{i\gamma}\otimes l_{j\gamma},\\
\mlabel{bialg-eq3}&&\sum_{i,j,\gamma} (y_{1i\gamma}\diamond x\otimes y_{2j\gamma} )
=\widehat{\tau}\Delta(x\diamond y)-\sum_{i,j, \gamma}g_{i\gamma}\otimes h_{j\gamma}+\Delta(y\diamond x)+\sum_{i,j,\gamma} k_{i\gamma}\otimes l_{j\gamma},\\
\mlabel{bialg-eq5}&&\sum_{i,j,\gamma} (x_{1i\gamma}\otimes y \diamond x_{2j\gamma})=\Delta(y\diamond x),\;\;\sum_{i,j,\gamma} (y \diamond x_{1i\gamma}\otimes x_{2j\gamma})=-\widehat{\tau} \Delta(y\diamond x)+\sum_{i,j,\gamma} g_{i\gamma}\otimes h_{j\gamma},\\
\mlabel{bialg-eq7}&&\sum_{i,j,\gamma} (x_{1i\gamma}\diamond y\otimes x_{2j\gamma})
=\widehat{\tau} \Delta(y\diamond x)-\sum_{i,j,\gamma}g_{i\gamma}\otimes h_{j\gamma}+\Delta(x\diamond y)+\sum_{i,j,\gamma} k_{i\gamma}\otimes l_{j\gamma},\\
\label{bialg-eq8}&&\sum_{i,j,\gamma} (x_{1i\gamma}\otimes x_{2j\gamma}\diamond y)=-\Delta(y\diamond x)-\sum_{i,j,\gamma} k_{i\gamma}\otimes l_{j\gamma},
\end{eqnarray}}
where $\sum_{i,j,\gamma} k_{i\gamma}\otimes l_{j\gamma}\in B\hatot  B$ is chosen so that $(\sum_{i,j,\gamma} k_{i\gamma}\otimes l_{j\gamma}, e\otimes f)=(y, e\diamond (f\diamond b))$. Applying Eqs. (\mref{bialg-eq1})-(\mref{bialg-eq8}) and Eqs. \eqref{lfd1}-\eqref{lfd8}, we obtain
{\small
\begin{flalign*}
&(\delta-\widehat{\tau}\delta)((a\ot x) \circ (b\ot y))-(L_{\circ}(a\ot x)\hatot \id)(\delta-\widehat{\tau}\delta)(b\ot y)-(\id\hatot L_{\circ}(a\ot x))(\delta-\widehat{\tau}\delta)(b\ot y)\\
&\quad\quad-(\id\hatot R_{\circ}(b\ot y))\delta(a\ot x)+(R_{\circ}(b\ot y)\hatot \id)\widehat{\tau}\delta(a\ot x)\\
&\quad=(\id_{L\widehat{\ot} L}-\widehat{\tau})(((\alpha+\beta)(a\vartriangleright b)-\sum_{(b)}(b_{(1)}\ot (b_{(2)}\vartriangleleft a+ a\vartriangleright b_{(2)})+b_{(2)}\ot (a\vartriangleright b_{(1)}+b_{(1)}\vartriangleleft a))\\
&\quad\quad-\sum_{[b]}(b_{[1]}\ot( a\vartriangleright b_{[2]}+b_{[2]}\vartriangleleft a)+b_{[2]}\ot (b_{[1]}\vartriangleleft a+a\vartriangleright b_{[1]}))
- \sum_{[a]}a_{[1]}\vartriangleright b\ot a_{[2]})\bullet \Delta(x\diamond y))\\
&\quad\quad+(\id_{L\widehat{\ot} L}-\widehat{\tau})(-(\alpha+\beta)(b\vartriangleleft a)+\sum_{(b)}b_{(1)}\vartriangleleft a\ot b_{(2)}
+ \sum_{[b]}b_{[1]}\vartriangleleft a\ot b_{[2]}\\
&\quad\quad+\sum_{(a)}a_{(1)}\ot (b\vartriangleleft a_{(2)}+a_{(2)}\vartriangleright b)
+ \sum_{[a]}a_{[2]}\ot (b\vartriangleleft a_{[1]}+a_{[1]}\vartriangleright b))\bullet \Delta(y\diamond x))\\
&\quad\quad+(\sum_{(b)}(-a\vartriangleright b_{(1)}\ot b_{(2)}+b_{(2)}\ot b_{(1)}\vartriangleleft a-b_{(1)}\vartriangleleft a\ot b_{(2)}+b_{(2)}\ot a\vartriangleright b_{(1)})\\
&\quad\quad+\sum_{[b]}(-a\vartriangleright b_{[1]}\ot b_{[2]}+b_{[2]}\ot b_{[1]}\vartriangleleft a-b_{[1]}\vartriangleleft a\ot b_{[2]}+b_{[2]}\ot a\vartriangleright b_{[1]})\\
&\quad\quad+\sum_{[a]}(b\vartriangleleft a_{[1]}\ot a_{[2]}+a_{[1]}\vartriangleright b\ot a_{[2]}-a_{[2]}\ot b\vartriangleleft a_{[1]}-a_{[2]}\ot  a_{[1]}\vartriangleright b))\bullet\sum_{i,j,\gamma}g_{i\alpha}\otimes h_{j\alpha}\\
&\quad\quad+(\sum_{(b)}(b_{(1)}\vartriangleleft a\ot b_{(2)}+ b_{(2)}\vartriangleleft a\ot b_{(1)}- b_{(1)}\ot b_{(2)}\vartriangleleft a- b_{(2)}\ot b_{(1)}\vartriangleleft a)\\
&\quad\quad+\sum_{[b]}(b_{[1]}\vartriangleleft a\ot b_{[2]}+ b_{[2]}\vartriangleleft a\ot b_{[1]}- b_{[1]}\ot b_{[2]}\vartriangleleft a- b_{[2]}\ot b_{[1]}\vartriangleleft a)\\
&\quad\quad+\sum_{(a)}(a_{(1)}\ot a_{(2)}\vartriangleright b-a_{(2)}\vartriangleright b\ot a_{(1)})+\sum_{[a]}(a_{[2]}\ot a_{[1]}\vartriangleright b-a_{[1]}\vartriangleright b\ot a_{[2]}))\bullet\sum_{i,j,\gamma}k_{i\alpha}\otimes l_{j\alpha}\\
&\quad=0,
\end{flalign*}}
and
\begin{flalign*}
&\delta((a\ot x)\circ (b\ot y)-(b\ot y)\circ (a\ot y))-(L_{\circ}(a\ot x)\hatot \id)\delta(b\ot x)+(L_{\circ}(b\ot y
)\hatot \id)\delta(a\ot x)\\
&\quad\quad-(\id\hatot (L_{\circ}(a\ot x)-R_{\circ}(a\ot x)))\delta( b\ot y)-(\id\hatot (R_{\circ}(b\ot y)-L_{\circ}(b\ot y)))\delta(a\ot x)\\
&\quad=(\beta(a\vartriangleright b+a\vartriangleleft b)-\sum_{(b)}(b_{(1)}\ot(a\vartriangleright b_{(2)} +a\vartriangleleft b_{(2)}+b_{(2)}\vartriangleleft a+b_{(2)}\vartriangleright a  ))\\
&\quad\quad-\sum_{[b]}(b_{[2]}\ot(a\vartriangleright b_{[1]}
+a\vartriangleleft b_{[1]}+b_{[1]}\vartriangleleft a+b_{[1]}\vartriangleright a  ))-\sum_{(a)}a_{(1)}\vartriangleleft b\ot a_{(2)})\bullet\Delta(x\diamond y)\\
&\quad\quad+(-\beta(b\vartriangleright a+b\vartriangleleft a)+\sum_{(a)}(a_{(1)}\ot(b\vartriangleright a_{(2)} +b\vartriangleleft a_{(2)}+a_{(2)}\vartriangleleft b+a_{(2)}\vartriangleright b  ))\\
&\quad\quad+\sum_{[a]}(a_{[2]}\ot(b\vartriangleright a_{[1]}
+b\vartriangleleft a_{[1]}+a_{[1]}\vartriangleleft b+a_{[1]}\vartriangleright b  ))+\sum_{(b)}b_{(1)}\vartriangleleft a\ot b_{(2)})\bullet\Delta(y\diamond x)\\
&\quad\quad+(-\tau\alpha(a\vartriangleright b+a\vartriangleleft b)+\sum_{(b)}(a\vartriangleright b_{(1)}+b_{(1)}\vartriangleleft a)\ot b_{(2)}+\sum_{[b]}(a\vartriangleright b_{[2]}+b_{[2]}\vartriangleleft a)\ot b_{[1]}\\
&\quad\quad+\sum_{[a]}a_{[2]}\ot (a_{[1]}\vartriangleright b+a_{[1]}\vartriangleleft b))\bullet \widehat{\tau}\Delta(x\diamond y)+(\tau\alpha(b\vartriangleright a+b\vartriangleleft a)-\sum_{(a)}(b\vartriangleright a_{(1)}\\
&\quad\quad+a_{(1)}\vartriangleleft b)\ot a_{(2)}-\sum_{[a]}(b\vartriangleright a_{[2]}+a_{[2]}\vartriangleleft b)\ot a_{[1]}-\sum_{[b]}b_{[2]}\ot (b_{[1]}\vartriangleright a\\
&\quad\quad+b_{[1]}\vartriangleleft a))\bullet \widehat{\tau}\Delta(y\diamond x)+(-\sum_{(b)}(a\vartriangleright b_{(1)}+b_{(1)}\vartriangleleft a)\ot b_{(2)}+\sum_{[b]}b_{[2]}\ot(a\vartriangleright b_{[1]}\\
&\quad\quad+a\vartriangleleft b_{[1]}+b_{[1]}\vartriangleleft a+b_{[1]}\vartriangleright a )+\sum_{(a)}(b\vartriangleright a_{(1)}+a_{(1)}\vartriangleleft b)\ot a_{(2)}-\sum_{[a]}a_{[2]}\ot(b\vartriangleright a_{[1]}\\
&\quad\quad+b\vartriangleleft a_{[1]}+a_{[1]}\vartriangleright b+a_{[1]}\vartriangleleft b))\bullet\sum_{i,j,\gamma}g_{i\gamma}\otimes h_{j\gamma}+(\sum_{(b)}(b_{(1)}\vartriangleleft a\ot b_{(2)}-b_{(1)}\ot(b_{(2)}\vartriangleleft a\\
&\quad\quad+b_{(2)}\vartriangleright a))+\sum_{[b]}(b_{[2]}\vartriangleleft a\ot b_{[1]}-b_{[2]}\ot(b_{[1]}\vartriangleleft a+b_{[1]}\vartriangleright a))+\sum_{(a)}(a_{(1)}\ot(a_{(2)}\vartriangleright b+a_{(2)}\vartriangleleft b)\\
&\quad\quad-a_{(1)}\vartriangleleft b\ot a_{(2)})+\sum_{[a]}(a_{[2]}\ot(a_{[1]}\vartriangleright b+a_{[1]}\vartriangleleft b)-a_{[2]}\vartriangleleft b\ot a_{[1]}))\bullet\sum_{i,j,\gamma}k_{i\gamma}\otimes l_{j\gamma}\\
&\quad=0.
\end{flalign*}
Therefore, $(L, \circ, \delta)$ is a \complete pre-Lie bialgebra.

If $(B, \diamond, (\cdot, \cdot))=({\bf k}[t,t^{-1}], \diamond, (\cdot, \cdot))$
 and $(L, \circ, \delta)$ is a \complete pre-Lie bialgebra, then
 $(A, \vartriangleleft,\vartriangleright)$ is a pre-Novikov algebra and $(A, \alpha,\beta)$ is a pre-Novikov coalgebra by Theorems \ref{t5} and \ref{t6}   respectively. Note that the linear map $\delta$ is given by Eq. \eqref{pre-Lie-coproduct}. Hence it suffices to  check Eqs. \eqref{lfd1}-\eqref{lfd8}.

By Eq. \eqref{eq:pre-Lbi1}, for all  $a, b\in A$, we compute
\begin{flalign*}
0=&(\delta-\widehat{\tau}\delta)(at^j\circ bt^k)-(L_{\circ}(at^j)\hatot \id)(\delta-\widehat{\tau}\delta)(bt^k)-(\id\hatot L_{\circ}(at^j))(\delta-\widehat{\tau}\delta)(bt^k)\\
&\quad-(\id\hatot R_{\circ}(bt^k))\delta(at^j)+(R_{\circ}(bt^k)\hatot \id)\widehat{\tau}\delta(at^j)\\
=&\sum_{i\in \mathbb{Z}}(j(i+1)(\sum_{(a\vartriangleright b)}(a\vartriangleright b)_{(1)}t^{-i-2}\ot (a\vartriangleright b)_{(2)}t^{i+j+k-1}-\sum_{[a\vartriangleright b]}(a\vartriangleright b)_{[2]}t^{i+j+k-1}\ot (a\vartriangleright b)_{[1]}t^{-i-2})\\
&-k(i+1)(\sum_{(b\vartriangleleft a)}(b\vartriangleleft a)_{(1)}t^{-i-2}\ot (b\vartriangleleft a)_{(2)}t^{i+j+k-1}-\sum_{[b\vartriangleleft a]}(b\vartriangleleft a)_{[2]}t^{i+j+k-1}\ot (b\vartriangleleft a)_{[1]}t^{-i-2})\\
&-j(i+1)(\sum_{(a\vartriangleright b)}((a\vartriangleright b)_{(2)}t^{i+j+k-1}\ot( a\vartriangleright b)_{(1)}t^{-i-2} -\sum_{[a\vartriangleright b]}(a\vartriangleright b)_{[1]}t^{-i-2}\ot (a\vartriangleright b)_{[2]}t^{i+j+k-1}) \\
&+k(i+1)(\sum_{(b\vartriangleleft a)}((b\vartriangleleft a)_{(2)}t^{i+j+k-1}\ot( b\vartriangleleft a)_{(1)}t^{-i-2} -\sum_{[b\vartriangleleft a]}(b\vartriangleleft a)_{[1]}t^{-i-2}\ot (b\vartriangleleft a)_{[2]}t^{i+j+k-1}))\\
&+\sum_{i\in \mathbb{Z}}((i+1)\sum_{(b)}(-ja \vartriangleright b_{(1)}t^{j-i-3}-(i+2)b_{(1)}\vartriangleleft a t^{j-i-3})\ot b_{(2)}t^{ k+i}\\
&+(i+1)\sum_{[b]}(ja \vartriangleright b_{[2]}t^{i+j+k-1}-(k+i)b_{[2]}\vartriangleleft a t^{i+j+k-1})\ot b_{[1]}t^ {-i-2}\\
&+(i+1)\sum_{(b)}(ja \vartriangleright b_{(2)}t^{i+j+k-1}-(k+i)b_{(2)}\vartriangleleft a t^{i+j+k-1})\ot b_{(1)}t^ {-i-2}\\
&+(i+1)\sum_{[b]}(-ja \vartriangleright b_{[1]}t^{j-i-3}-(i+2)b_{[1]}\vartriangleleft a t^{j-i-3})\ot b_{[2]}t^{ k+i})\\
&+\sum_{i\in \mathbb{Z}}((i+1)\sum_{(b)}b_{(1)}t^{-i-2}\ot (-ja \vartriangleright b_{(2)}t^{i+j+k-1}+(k+i)b_{(2)}\vartriangleleft at^{i+j+k-1})\\
&+(i+1)\sum_{[b]}b_{[2]}t^{k+i}\ot (ja \vartriangleright b_{[1]}t^{j-i-3}+(i+2)b_{[1]}\vartriangleleft at^{j-i-3})\\
&+(i+1)\sum_{(b)}b_{(2)}t^{k+i}\ot (ja \vartriangleright b_{(1)}t^{j-i-3}+(i+2)b_{(1)}\vartriangleleft at^{j-i-3})\\
&+(i+1)\sum_{[b]}b_{[1]}t^{-i-2}\ot (-ja \vartriangleright b_{[2]}t^{i+j+k-1}+(k+i)b_{[2]}\vartriangleleft at^{i+j+k-1}))\\
&+\sum_{i\in \mathbb{Z}}((i+1)(\sum_{(a)} a_{(1)}t^{-i-2}\ot(-(i+j)a_{(2)}\vartriangleright bt^{i+j+k-1}+kb \vartriangleleft a_{(2)}t^{i+j+k-1})\\
&+(i+1)(\sum_{[a]} a_{[2]}t^{i+j}\ot(-(i+2)a_{[1]}\vartriangleright bt^{k-i-3}-kb \vartriangleleft a_{[1]}t^{k-i-3}))\\
&+\sum_{i\in \mathbb{Z}}((i+1)(\sum_{(a)}((i+j)a_{(2)}\vartriangleright bt^{i+j+k-1}-kb \vartriangleleft a_{(2)}t^{i+j+k-1}) \ot a_{(1)}t^{-i-2}\\
&+(i+1)(\sum_{[a]}((i+2)a_{[1]}\vartriangleright bt^{k-i-3}+kb \vartriangleleft a_{[1]}t^{k-i-3})\ot a_{[2]}t^{i+j}).
\end{flalign*}

Setting $j=0$, $k=2$ and comparing the coefficients of $1\otimes t^{-1}$, we obtain
\begin{eqnarray*}
(\alpha+\beta)(b\vartriangleleft a)=(\id\otimes(R_{\vartriangleright}+L_{\vartriangleleft})(b))(\tau\alpha+\beta)(a)+(R_{\vartriangleleft}(a)\otimes\id)(\alpha+\beta)(b).
\end{eqnarray*}
Then we obtain Eq. \eqref{lfd8}.
Similarly, setting $j=1,k=-1$ and comparing the coefficients of $t^{-2}\otimes t^{-1}$ yield Eq. \eqref{lfd3};
setting $j=0,k=-1$ and comparing the coefficients of $t^{-2}\otimes t^{-2}$ yield Eq. \eqref{lfd4};
setting $j=-1,k=2$ and comparing the coefficients of $t^{-1}\otimes t^{-1}$ yield Eq. \eqref{lfd7}.

By Eq. \eqref{eq:pre-Lbi2}, for all  $a, b\in A$, we compute
{\small
\begin{flalign*}
0=&\delta(at^j\circ bt^k-bt^k\circ at^j)-(L_{\circ}(at^j)\hatot \id)\delta(bt^k)+(L_{\circ}(bt^k
)\hatot \id)\delta(at^j)\\
&\quad-(\id\hatot (L_{\circ}(at^j)-R_{\circ}(at^j)))\delta(bt^k)-(\id\hatot (R_{\circ}(bt^k)-L_{\circ}(bt^k)))\delta(at^j)\\
=&\sum_{i\in \mathbb{Z}}\big(j(i+1)(\sum_{(a\vartriangleright b)}(a\vartriangleright b)_{(1)}t^{-i-2}\ot (a\vartriangleright b)_{(2)}t^{i+j+k-1}-\sum_{[a\vartriangleright b]}(a\vartriangleright b)_{[2]}t^{i+j+k-1}\ot (a\vartriangleright b)_{[1]}t^{-i-2})\\
&-k(i+1)(\sum_{(b\vartriangleleft a)}(b\vartriangleleft a)_{(1)}t^{-i-2}\ot (b\vartriangleleft a)_{(2)}t^{i+j+k-1}-\sum_{[b\vartriangleleft a]}(b\vartriangleleft a)_{[2]}t^{i+j+k-1}\ot (b\vartriangleleft a)_{[1]}t^{-i-2})\\
&-k(i+1)(\sum_{(b\vartriangleright a)}(b\vartriangleright a)_{(1)}t^{-i-2}\ot (b\vartriangleright a)_{(2)}t^{i+j+k-1}-\sum_{[b\vartriangleright a]}(b\vartriangleright a)_{[2]}t^{i+j+k-1}\ot (b\vartriangleright a)_{[1]}t^{-i-2})\\
&+j(i+1)(\sum_{(a\vartriangleleft b)}(a\vartriangleleft b)_{(1)}t^{-i-2}\ot (a\vartriangleleft b)_{(2)}t^{i+j+k-1}-\sum_{[a\vartriangleleft b]}(a\vartriangleleft b)_{[2]}t^{i+j+k-1}\ot (a\vartriangleleft b)_{[1]}t^{-i-2})\big)\\
&+\sum_{i\in \mathbb{Z}}\big((i+1)\sum_{(a)}a_{(1)}t^{-i-2}\ot (-(i+j)a_{(2)}\vartriangleright bt^{i+j+k-1}+kb\vartriangleleft a_{(2)}t^{i+j+k-1})\\
&+(i+1)\sum_{[a]}a_{[2]}t^{i+j}\ot (-(i+2)a_{[1]}\vartriangleright bt^{k-i-3}-kb\vartriangleleft a_{[1]}t^{k-i-3})\\
&+(i+1)\sum_{(a)}a_{(1)}t^{-i-2}\ot(kb\vartriangleright a_{(2)}t^{i+j+k-1}-(i+j)a_{(2)}\vartriangleleft bt^{i+j+k-1})\\
&+(i+1)\sum_{[a]}a_{[2]}t^{i+j}\ot (-kb\vartriangleright a_{[1]}t^{k-i-3}-(i+2)a_{[1]}\vartriangleleft bt^{k-i-3})\big)\\
&+\sum_{i\in \mathbb{Z}}\big((i+1)\sum_{(b)}b_{(1)}t^{-i-2}\ot ((k+i)b_{(2)}\vartriangleright at^{i+j+k-1}-ja\vartriangleleft b_{(2)}t^{i+j+k-1})\\
&+(i+1)\sum_{[b]}b_{[2]}t^{k+i}\ot ((i+2)b_{[1]}\vartriangleright at^{j-i-3}+ja\vartriangleleft b_{[1]}t^{j-i-3})\\
&+(i+1)\sum_{(b)}b_{(1)}t^{-i-2}\ot (-j a\vartriangleright b_{(2)}t^{i+j+k-1}+(k+i) b_{(2)}\vartriangleleft a t^{i+j+k-1})\\
&+(i+1)\sum_{[b]}b_{[2]}t^{k+i}\ot (j a\vartriangleright b_{[1]}t^{j-i-3}+(i+2) b_{[1]}\vartriangleleft a t^{j-i-3})\big)\\
&+\sum_{i\in \mathbb{Z}}\big((i+1)\sum_{(b)}(-ja\vartriangleright b_{(1)}t^{j-i-3}-(i+2)b_{(1)}\vartriangleleft a t^{j-i-3})\ot b_{(2)}t^{k+i}\\
&+(i+1)\sum_{[b]}(ja\vartriangleright b_{[2]}t^{i+j+k-1}-(k+i)b_{[2]}\vartriangleleft a t^{i+j+k-1})\ot b_{[1]}t^{-i-2}\big)\\
&\sum_{i\in \mathbb{Z}}\big((i+1)\sum_{(a)}(kb\vartriangleright a_{(1)}t^{k-i-3}+(i+2)a_{(1)}\vartriangleleft b t^{k-i-3})\ot a_{(2)}t^{i+j}\\
&+(i+1)\sum_{[a]}(-kb\vartriangleright a_{[2]}t^{i+j+k-1}+(i+j)a_{[2]}\vartriangleleft b t^{i+j+k-1})\ot a_{[1]}t^{-i-2}\big).
\end{flalign*}}

Setting $j=1,k=0$ and comparing the coefficients of $t^{-2}\otimes 1$ yield Eq. \eqref{lfd1};
setting $j=k=1$ and comparing the coefficients of $t^{-1}\otimes 1$ yield Eq. \eqref{lfd2};
setting $j=k=0$ and comparing the coefficients of $t^{-3}\otimes 1$ yield Eq. \eqref{lfd5};
setting $j=2,k=0$ and comparing the coefficients of $t^{-1}\otimes 1$ yield Eq. \eqref{lfd6}.

This completes the proof.
\end{proof}

 Finally, we present two examples of  completed pre-Lie bialgebras.
\begin{example}\label{ex-pL-bialg1}
Let $(A,\vartriangleleft,\vartriangleright,\alpha,\beta)$ be the pre-Novikov bialgebra given in Example \ref{ex1} and $(B, \diamond, (\cdot, \cdot))=({\bf k}[t,t^{-1}], \diamond, (\cdot, \cdot))$ be the quadratic $\mathbb{Z}$-graded right Novikov algebra given in Example \ref{Laurent-Bilinear}. Then by Theorem \ref{thm-bi}, there is a \complete pre-Lie bialgebra $(A\ot B,\circ,\delta)$ which is defined by (for all $i,j,k\in \mathbb{Z}$)
\begin{flalign*}
e_1t^i\circ e_1t^j&=-je_1t^{i+j-1},e_1t^i\circ e_2t^j=e_2t^i\circ e_1t^j=-je_2t^{i+j-1},\quad e_2t^i\circ e_2t^j=0,\\
\delta(e_1t^k)&=\sum_{i\in \mathbb{Z}}(i+1)(-e_2t^{-i-2}\ot e_2t^{k+i}-e_2t^{k+i}\ot e_2t^{-i-2} ),\;\;\quad\quad \delta(e_2t^k)=0.
\end{flalign*}

Let  $(C,\tilde{\diamond})$ be the quadratic  right Novikov algebra given in Example \ref{quadratic-ex}. Then by Theorem \ref{thm-bi}, there is a  pre-Lie bialgebra $(A\ot C,\tilde{\circ},\tilde{\delta})$, where $\tilde{\circ}$ is defined by non-zero products:
\begin{flalign*}
	&(e_1\otimes y)\tilde{\circ}(e_1\ot y)=-e_1\ot y, \quad	(e_1\otimes x)\tilde{\circ}(e_1\ot y)=-e_1\ot x,\\
	&(e_1\otimes x)\tilde{\circ}(e_2\ot y)=(e_2\otimes x)\tilde{\circ}(e_1\ot y)=-e_2\otimes x,\;\; (e_1\otimes y)\tilde{\circ}(e_1\ot x)=2e_1\ot x,\\
	&(e_1\otimes y)\tilde{\circ}(e_2\ot x)=(e_2\otimes y)\tilde{\circ}(e_1\ot x)=2e_2\otimes x,\\
	&(e_1\otimes y)\tilde{\circ}(e_2\ot y)=	(e_2\otimes y)\tilde{\circ}(e_1\ot y)=-e_2\otimes y,
\end{flalign*}
and $\tilde{\delta}$ is defined by
\begin{flalign*}
	& \tilde{\delta}(e_1\otimes x)=-2(e_2\otimes x)\otimes (e_2\ot x),\quad \tilde{\delta}(e_2\otimes x)=0,\\
	&\tilde{\delta}(e_1\ot y)=(e_2\ot x)\otimes(e_2\otimes y)+(e_2\otimes y)\otimes(e_2\otimes x),\quad \tilde{\delta}(e_2\ot y)=0.
\end{flalign*}

\end{example}

\section{Completed pre-Lie bialgebras from  pre-Novikov Yang-Baxter equation}
In this section, we construct symmetric completed solutions of the $S$-equation from symmetric solutions of  the pre-Novikov Yang-Baxter equation.

First, we introduce some notations.
Let $(A,\vartriangleleft,\vartriangleright)$ be a pre-Novikov algebra and $(A,\circ)$ be the associated Novikov algebra. For convenience, we define binary operations $\odot$ and $\star$ on $A$ by
$$a\odot b:=a\vartriangleright b+b\vartriangleleft a,\qquad a\star b:=a\circ b+b\circ a, \quad  a,b \in A.$$
Let $V$ be a vector space with a binary operation $\ast$. Let $r=\sum\limits_{i}x_i\otimes y_i \in V\otimes V$ and $r^{'}=\sum\limits_{i}x_i^{'}\otimes y_i^{'} \in V\otimes V$. Set
\begin{eqnarray*}
r_{12}\ast r_{13}^{'}:=\sum_{i,j}x_i\ast x_j^{'}\otimes y_i\otimes y_j^{'},\;r_{13}\ast r_{23}^{'}:=\sum_{i,j}x_i\otimes x_j^{'}\otimes y_i\ast y_j^{'},\\
r_{12}\ast r_{23}^{'}:=\sum_{i,j}x_i\otimes y_i\ast x_j^{'} \otimes y_j^{'},\;r_{23}\ast r_{13}^{'}:=\sum_{i,j}x_j^{'}\otimes x_i\otimes y_i\ast y_j^{'}.\\
\end{eqnarray*}
\begin{definition}\cite{LH}
Let $(A,\vartriangleleft,\vartriangleright)$ be a pre-Novikov algebra, $r\in A\otimes A$ and $(A,\circ)$ be the associated Novikov algebra of $(A,\vartriangleleft,\vartriangleright)$. The following equation
\begin{flalign}
r_{12}\circ r_{13}+r_{23}\odot r_{13}-r_{12}\vartriangleleft r_{23}=0\label{yb}
\end{flalign}
is called the \textbf{pre-Novikov Yang-Baxter equation (p-NYBE)} in $(A,\vartriangleleft,\vartriangleright)$.
\end{definition}

\begin{proposition}\cite{LH}\label{pN-coboundary}
Let $(A,\vartriangleleft,\vartriangleright)$ be a pre-Novikov algebra and $r\in A\otimes A$ be a symmetric solution of p-NYBE in $(A,\vartriangleleft,\vartriangleright)$. Let $\alpha,\beta:A\rightarrow A\otimes A$ be linear maps defined by
\begin{flalign}
\alpha(a):&=(L_{\circ}(a)\otimes \id+\id \otimes(L_{\vartriangleright}+R_{\vartriangleleft})(a))\tau r,\label{e11}\\
\beta(a):&=-(L_{\vartriangleright}(a)\otimes \id+\id \otimes (L_{\circ}+R_{\circ})(a))r, \quad a\in A.\label{e12}
\end{flalign}
Then $(A,\vartriangleleft,\vartriangleright,\alpha,\beta)$ is a pre-Novikov bialgebra.
\end{proposition}

Let ($L=\oplus_{i\in \ZZ}L_i ,\circ )$ be a $\ZZ$-graded pre-Lie algebra.
Suppose that $r=\sum_{i,j,\alpha}a_{i\alpha}\otimes b_{j\alpha}\in L\hatot L$ as in Eq.~\meqref{eq:ssum}.
We denote
\begin{eqnarray*}
&&r_{12}\circ r_{13}:=\!\sum_{i,j,k,l,\alpha, \beta}a_{i\alpha}\circ a_{k\beta}\otimes b_{j\alpha}\otimes b_{l\beta},\;r_{12}\circ r_{23}:=\!\sum_{i,j,k,l,\alpha, \beta}a_{i\alpha}\otimes b_{j\alpha}\circ a_{k\beta}\otimes b_{l\beta},\\
&& [r_{13}, r_{23}]:=\!\sum_{i,j,k,l,\alpha,\beta} a_{i\alpha}\otimes a_{k\beta}\otimes (b_{j\alpha}\circ b_{l\beta}-b_{l\beta}\circ b_{j\alpha}),
\end{eqnarray*}
provided the sums make sense. Note that these sums make sense when $r=\sum_{i\in \ZZ, \alpha} c_{i,\alpha}\otimes d_{-i-s,\alpha}\in L\widehat{\otimes} L $ for some fixed $s\in \ZZ$, which is the case that we are interested in the sequel.

\begin{definition}
Let ($L=\oplus_{i\in \ZZ}L_i ,\circ )$ be a $\ZZ$-graded pre-Lie algebra.
If $r\in L\hatot L$ satisfies the {\bf $S$-equation } (see \cite{Bai1})
$$-r_{12}\circ r_{13}+r_{12}\circ r_{23}+[r_{13}, r_{23}]=0
$$
as an element in $L \hatot L \hatot L$, then $r$ is called a {\bf completed solution of the  $S$-equation} in $L$.
\end{definition}

Next, we present a construction of symmetric completed solutions of the  $S$-equation from symmetric solutions of the p-NYBE.
\begin{theorem}\mlabel{pro:p-NYBE}
Let $(A,\vartriangleleft,\vartriangleright)$ be a pre-Novikov algebra and $\big(B=\oplus_{i\in
\ZZ}B_i,\diamond, (\cdot, \cdot)\big)$ be a quadratic $\ZZ$-graded
right  Novikov algebra. Let $(L:=A\otimes B,\circ)$ be the induced pre-Lie
algebra from $(A,
\vartriangleleft,\vartriangleright)$ and $\big(B=\oplus_{i\in
	\ZZ}B_i,\diamond, (\cdot, \cdot)\big)$ in Theorem \ref{t5}. Suppose that $r=\sum_\alpha x_\alpha\otimes y_\alpha\in
A\otimes A$ is a symmetric solution of p-NYBE in $A$.
Then for a basis $\{e_p\}_{p\in \Pi}$ consisting of homogeneous
elements of $B$ and its homogeneous dual basis $\{f_p\}_{p\in
\Pi}$ associated with the bilinear form $(\cdot, \cdot)$, the
tensor element
\begin{eqnarray}\label{eq:S-e}
r_L\coloneqq \sum_{p\in \Pi} \sum_\alpha
(x_\alpha\otimes e_p)\otimes (y_\alpha\otimes f_p) \in
L \hatot L
\end{eqnarray}
is a symmetric completed solution of the $S$-equation in $L$. Furthermore, if the quadratic $\mathbb{Z}$-graded right Novikov algebra is
$(B, \diamond, (\cdot,\cdot))$ $=({\bf k}[t,t^{-1}], \diamond,
(\cdot, \cdot))$ from Example~\mref{Laurent-Bilinear}, then
\begin{eqnarray}\label{eq:affine} r_L=\sum_{ i\in
\mathbb{Z}}\sum_\alpha x_\alpha t^i\,\otimes \,y_\alpha t^{-i-1}\in L
\hatot  L
\end{eqnarray}
is a symmetric completed solution of the $S$-equation in $L$ if and
only if $r$ is a symmetric solution of  the p-NYBE in $A$.
\end{theorem}
%If $B$ is a finite-dimensional right Novikov algebra, then Eq. \eqref{eq:S-e} is a finite sum.

\begin{proof} Applying the notation in Eq.~\eqref{eq:pairb}, we have
\begin{eqnarray*}
\Big(e_q\otimes e_s, \sum_{p\in \Pi}e_p\otimes
f_p\Big)=\sum_{p\in \Pi}(e_q, e_p)(e_s,f_p)=(e_q, e_s).
\end{eqnarray*}
Due to the bilinear form $(\cdot, \cdot)$ on $B$ is symmetric and nondegenerate, we obtain
$\sum_{p\in \Pi}e_p\otimes f_p=\sum_{p\in
\Pi}f_p\otimes e_p$. Hence, $r_L$ is also symmetric. Therefore we have
\begin{align*}
&\quad-r_{12}\circ r_{13}+r_{12}\circ r_{23}+[r_{13}, r_{23}]\\
&=\sum_{p,q\in \Pi}\sum_{\alpha,\beta} \Big( (-x_\alpha \vartriangleright x_\beta \ot e_p\diamond e_q+x_\beta\vartriangleleft x_\alpha \ot e_q\diamond e_p)\ot (y_\alpha\ot f_p)\ot( y_\beta\ot f_q)\\
&\quad+(x_\alpha\ot e_p)\ot(y_\alpha\vartriangleright x_\beta\ot f_p\diamond e_q-x_\beta\vartriangleleft y_\alpha\ot e_q\diamond f_p)\ot(y_\beta\ot f_q)\\
&\quad+(x_\alpha\ot e_p)\ot(x_\beta\ot e_q)\ot(y_\alpha\vartriangleright y_\beta\ot f_p\diamond f_q-y_\beta\vartriangleleft y_\alpha\ot f_q\diamond f_p\\
&\quad-y_\beta \vartriangleright y_\alpha\ot f_q\diamond f_p+y_\alpha\vartriangleleft y_\beta\ot f_p\diamond f_q)\Big).
\end{align*}
For $s$, $u$, $v\in \Pi$, adopting the notation in Eq.~\eqref{eq:pairb} we have
 \begin{eqnarray*}
	&\Big(e_s\otimes e_u\otimes e_v, \sum\limits_{p,q\in \Pi} e_p\diamond e_q\otimes f_p\otimes f_q\Big)=(e_s, e_u\diamond e_v),\\
	&\Big(e_s\otimes e_u\otimes e_v, \sum\limits_{p,q\in \Pi} e_q\diamond e_p\otimes f_p\otimes f_q\Big)=(e_s, e_v\diamond e_u),\\
	&\Big(e_s\otimes e_u\otimes e_v, \sum\limits_{p,q\in \Pi} e_p\otimes
	f_p\diamond  e_q\otimes f_q\Big)=-(e_s, e_u\diamond e_v+e_v\diamond  e_u).
\end{eqnarray*}
Then by the nondegenerate property of $(\cdot,\cdot)$ on $B$,
we obtain
\begin{eqnarray*}
\sum_{p,q\in \Pi}e_p\otimes f_p\diamond e_q\otimes
f_q=-\sum_{p,q\in \Pi}(e_p\diamond e_q\otimes  f_p\otimes
f_q+e_q\diamond e_p\otimes f_p\otimes f_q).
\end{eqnarray*}
Similarly, we get
\begin{eqnarray*}
&&\sum_{p,q\in \Pi} e_p\otimes e_q\otimes f_p\diamond f_q=-\sum_{p,q\in \Pi}(e_p\otimes e_q\diamond f_p\otimes f_q+e_p\otimes e_q\otimes f_q\diamond f_p),\\
&&\sum_{p,q\in \Pi}e_p\otimes e_q\diamond f_p\otimes
f_q=\sum_{p,q\in \Pi} e_q\diamond e_p\otimes  f_p\otimes
f_q,\\
&&\sum_{p,q\in \Pi} e_p\otimes e_q\otimes
f_q\diamond f_p=\sum_{p,q\in \Pi} e_p\diamond e_q\otimes
f_p\otimes f_q.
\end{eqnarray*}
Then we have
\begin{align*}
&\quad-r_{12}\circ r_{13}+r_{12}\circ r_{23}+[r_{13}, r_{23}]\\
&=\sum_{p,q\in \Pi}\sum_{\alpha,\beta} \Big( (x_\beta\vartriangleleft x_\alpha\ot e_q\diamond e_p)\ot( y_\alpha\ot f_p)\ot(y_\beta\ot f_q)\\
&\quad-(x_\alpha\ot e_q\diamond e_p)\ot(x_\beta\vartriangleleft y_\alpha \ot f_p)\ot(y_\beta\ot f_q)
-(x_\alpha\ot e_q\diamond e_p)\ot( y_\alpha\vartriangleright x_\beta\ot f_p)\ot(y_\beta\ot f_q)\\
&\quad-(x_\alpha\ot e_q\diamond e_p)\ot( x_\beta\ot f_p)\ot(y_\alpha\vartriangleright y_\beta\ot f_q)
-(x_\alpha\ot e_q\diamond e_p)\ot( x_\beta\ot f_p)\ot(y_\alpha\vartriangleleft y_\beta\ot f_q)\Big)\\
&\quad-\sum_{p,q\in \Pi}\sum_{\alpha,\beta} \Big((x_\alpha \vartriangleright x_\beta\ot e_p\diamond e_q)\ot(y_\alpha \ot f_p)\ot(y_\beta\ot f_q)\\
&\quad+(x_\alpha\ot e_p\diamond e_q)\ot(y_\alpha\vartriangleright x_\beta \ot f_p)\ot(y_\beta\ot f_q)+(x_\alpha\ot e_p\diamond e_q)\ot( x_\beta\ot f_p)\ot(y_\alpha\vartriangleright y_\beta\ot f_q)\\
&\quad+(x_\alpha\ot e_p\diamond e_q)\ot( x_\beta\ot f_p)\ot( y_\beta\vartriangleright y_\alpha\ot f_q)+(x_\alpha\ot e_p\diamond e_q)\ot( x_\beta\ot f_p)\ot(y_\alpha\vartriangleleft y_\beta\ot f_q)\\
&\quad+(x_\alpha\ot e_p\diamond e_q)\ot(x_\beta \ot f_p)\ot( y_\beta\vartriangleleft y_\alpha\ot f_q)\Big)\\
&=(r_{13}\vartriangleleft r_{12}-r_{12}\odot r_{23}-r_{13}\vartriangleright r_{23}-r_{13}\vartriangleleft r_{23})\bullet(e_q\diamond e_p\ot f_p\ot f_q)\\
&\quad-(r_{12}\vartriangleright r_{13}+r_{12}\vartriangleright r_{23}+r_{13}\star r_{23})\bullet (e_p\diamond e_q\ot f_p\ot f_q)\\
&=0.
\end{align*}
\delete{\textcolor{blue}{I added content in red, but it may not be necessary.}\\
\textcolor{red}{By Eq. \eqref{yb} and  the same argument as  in the proof of \cite[Theorem 4.4]{LH}, we obtain
\begin{align*}
&r_{13}\vartriangleleft r_{12}-r_{12}\odot r_{23}-r_{13}\vartriangleright r_{23}-r_{13}\vartriangleleft r_{23}=0,\\
&r_{12}\vartriangleright r_{13}+r_{12}\vartriangleright r_{23}+r_{13}\star r_{23}=0.
\end{align*}
Therefore, $-r_{12}\circ r_{13}+r_{12}\circ r_{23}+[r_{13}, r_{23}]=0$.}} Thus, $r_L$ is a symmetric completed
solution of  S-equation in $L$.

Moreover, if $(B, \diamond, (\cdot, \cdot))=({\bf k}[t,t^{-1}],
\diamond, (\cdot, \cdot))$, then we note that $\{t^{-i-1}|i\in \mathbb{Z}\}$ is the basis of ${\bf k}[t, t^{-1}]$ dual to $\{t^{i}|i\in
\mathbb{Z}\}$ associated with the bilinear form defined by $(t^i, t^j)=\delta_{i+j+1,0}$ for
all $i$, $j\in \mathbb{Z}$. Suppose that $r_L$ given by Eq. \eqref{eq:affine} is a symmetric completed
solution of  the $S$-equation in $L$.
Setting $i=0$, we conclude
that $r$ is symmetric since $\widehat{\tau }r_L=\sum_{ i\in
	\mathbb{Z}}\sum_\alpha y_\alpha t^{-i-1}\otimes x_\alpha t^{i} =\sum_{
	j\in \mathbb{Z}} \sum_\alpha y_\alpha t^{j}\otimes x_\alpha t^{-j-1}$.
%Since $\widehat{\tau }r_L=\sum_{ i\in
%\mathbb{Z}}\sum_\alpha y_\alpha t^{-i-1}\otimes x_\alpha t^{i} =\sum_{
%j\in \mathbb{Z}} \sum_\alpha y_\alpha t^{j}\otimes x_\alpha t^{-j-1}$, we conclude
%that $r$ is symmetric by setting $i=0$.
We compute
\begin{align*}
&0=r_{12}\circ r_{13}+r_{12}\circ r_{23}+[r_{13}, r_{23}]\\
&\ \ =\sum_{i,j\in \mathbb{Z}}\sum_{\alpha,\beta}\Big((jx_\beta\vartriangleleft x_\alpha t^{i+j-1}-ix_\alpha\vartriangleright x_\beta t^{i+j-1})\ot y_\alpha t^{-i-1}\ot y_\beta t^{-j-1}\\
&\quad+x_\alpha t^i\ot(-(i+1)y_\alpha\vartriangleright x_\beta t^{j-i-2}-jx_\beta\vartriangleleft y_\alpha t^{j-i-2})\ot y_\beta t^{-j-1}\\
&\quad+x_\alpha t^i\ot x_\beta t^j\ot(-(i+1)y_\alpha\vartriangleright y_\beta t^{-i-j-3}+(j+1)y_\beta\vartriangleleft y_\alpha t^{-i-j-3}\\
&\quad+(j+1)y_\beta\vartriangleright y_\alpha t^{-i-j-3}-(i+1)y_\alpha\vartriangleleft y_\beta t^{-i-j-3})\Big).
\end{align*}
Comparing the coefficients of $1\otimes t^{-1}\otimes t^{-2}$
yields that $r$ is a symmetric solution of the p-NYBE in
the pre-Novikov algebra $A$.

This completes the proof.
\end{proof}

\begin{example}
Let $(A=\mathbf{k}e_1\oplus \mathbf{k}e_2\oplus \mathbf{k}e_3\oplus \mathbf{k}e_4,\vartriangleleft,\vartriangleright)$ be a 4-dimensional vector space with binary operations
$\vartriangleleft,\vartriangleright$  defined by non-zero products
\begin{flalign*}
	&e_1\vartriangleleft e_1=e_1,\quad e_1\vartriangleleft e_2= e_2\vartriangleleft e_1=e_2,\\
	&e_1\vartriangleright e_3=e_2\vartriangleright e_4=-2e_3,\qquad e_1\vartriangleright e_4=-2e_4,\\
	&e_1\vartriangleleft e_3=e_2\vartriangleleft e_4=e_3\vartriangleleft e_1=e_4\vartriangleleft e_2=e_3,\quad e_1\vartriangleleft e_4=e_4\vartriangleleft e_1=e_4.
\end{flalign*}
One can directly check that $(A,\vartriangleleft,\vartriangleright)$ is a pre-Novikov algebra and $r:= e_2\otimes e_3+e_3\otimes e_2$ is a symmetric solution of  the p-NYBE in the pre-Novikov algebra $(A,\vartriangleleft,\vartriangleright)$.

Let $(B, \diamond, (\cdot, \cdot))=({\bf k}[t,t^{-1}], \diamond, (\cdot, \cdot))$ be the quadratic $\mathbb{Z}$-graded right Novikov algebra given in Example \ref{Laurent-Bilinear}.
Then by Theorem \ref{pro:p-NYBE}, $r_L:=\sum_{ i\in
	\mathbb{Z}}\big( (e_2\ot t^i)\otimes( e_3\ot t^{-i-1})+(e_3\ot t^i)\otimes( e_2\ot t^{-i-1})  \big)$ is a symmetric completed solution of  the $S$-equation in $A\ot B$.
	
Let  $(C,\tilde{\diamond})$ be the quadratic  right Novikov algebra given in Example \ref{quadratic-ex}. Then by Proposition \ref{pro:p-NYBE}, $\tilde{r_L}:= (e_2\ot x)\otimes( e_3\ot y)+(e_3\ot x)\otimes( e_2\ot y)+(e_2\ot y)\otimes( e_3\ot x)+(e_3\ot y)\otimes( e_2\ot x)  $ is a symmetric solution of the $S$-equation in $A\ot C$.
\end{example}

\begin{proposition}\mlabel{pL-coboundary}
Let $(A, \circ)$ be a \complete pre-Lie algebra and $r\in A\otimes A$ be a symmetric completed solution of the $S$-equation in $(A,\circ)$. Let $\delta:A\rightarrow A\otimes A$ be a linear map defined by
\begin{equation} \label{eq:rdelta}
\delta(a):=(L_{\circ}(a)\hatot  \id+\id \hatot (L_{\circ}(a)-R_{\circ}(a)))r,\quad a\in A.
\end{equation}
 Then $(A,\circ,\delta)$ is a \complete pre-Lie bialgebra.
\end{proposition}
\begin{proof}
The proof follows from the same argument as presented in \cite[proposition 6.1]{Bai1}.
\end{proof}

\begin{proposition}
Under the same assumption as in Proposition~\ref{pro:p-NYBE}, let $\alpha_r,\beta_r:A\rightarrow A\ot A$  be linear maps defined by Eq. \eqref{e11} and Eq. \eqref{e12} with $r\in A\ot A$ and $\delta:L\rightarrow L\hatot L$ be a linear map defined by Eq. \eqref{co-dipL}. Then $(A,\vartriangleleft,\vartriangleright,\alpha_r,\beta_r)$ is a pre-Novikov bialgebra and hence $(L,\circ ,\delta)$ is a completed pre-Lie bialgebra by Theorem \ref{thm-bi}. It coincides with the completed pre-Lie bialgebra with $\delta$ defined by Eq. \eqref{eq:rdelta} through $r_L$ by Proposition \ref{pro:p-NYBE}, where $r_L$ is defined by Eq. \eqref{eq:S-e}. That is, the following diagram commutes:
$$  \xymatrix{
\text{ symmetric solutions}\atop \text{of p-NYBE} \ar[rr]^-{\rm Prop. \ref{pN-coboundary}}   \ar[d]_-{\rm Thm. \mref{pro:p-NYBE}}&& \text{pre-Novikov}\atop \text{ bialgebras} \ar[d]_-{\rm Thm. \mref{thm-bi}}\\
\text{symmetric solutions}\atop \text{of S-equations} \ar[rr]^-{\rm
Prop. \mref{pL-coboundary}}
            && \text{pre-Lie}\atop \text{ bialgebras}  }
$$
\end{proposition}

\begin{proof}
Obviously, $(A,\vartriangleleft,\vartriangleright,\alpha_r,\beta_r)$ is a pre-Novikov bialgebra by Proposition \ref{pN-coboundary}.
By Theorem \ref{thm-bi}, there is a \complete pre-Lie
bialgebra structure $(L, \circ, \delta)$ on
$L$ where $\delta$ is induced from $\alpha_r,\beta_r$ by Eq.
(\mref{co-dipL}). For all $a\in A$, $b\in B$, we have
\begin{eqnarray*}
\delta(a\otimes b)&=& \sum_{i,j,\beta}\sum_\alpha(-(a\vartriangleright x_\alpha\otimes b_{1i\beta})\otimes (y_\alpha\otimes b_{2j\beta})-(x_\alpha\otimes b_{1i\beta})\otimes (a\star y_\alpha\otimes
b_{2j\beta})\\
&&-(x_\alpha\otimes b_{2j\beta})\otimes (a\circ y_\alpha\otimes
b_{1i\beta})-(a\odot x_\alpha\otimes b_{2j\beta})\otimes (y_\alpha\otimes
b_{1i\beta})).
\end{eqnarray*}
On the other hand,
{\small \begin{eqnarray*}
		&&(L_{\circ}(a\otimes b)\hatot  \id+\id \hatot ( L_{\circ}(a\otimes b)-R_{\circ}(a\otimes b)))\sum_{p\in \Pi}\sum_\alpha (x_\alpha\otimes e_p)\otimes (y_\alpha \otimes f_p)\\
		&=&\sum_{p\in \Pi}\sum_\alpha\Big((a\vartriangleright x_\alpha\otimes b\diamond e_p-x_\alpha\vartriangleleft
		a\otimes e_p\diamond b)\otimes (y_\alpha\otimes f_p) \\
		&&\qquad \quad +(x_\alpha\otimes
		e_p)\otimes (a\vartriangleright y_\alpha\otimes b\diamond f_p-y_\alpha\vartriangleleft
		a\otimes f_p\diamond b-y_\alpha\vartriangleright
		a\otimes f_p\diamond b+a\vartriangleleft y_\alpha\otimes b\diamond f_p)\Big).
\end{eqnarray*}}
For the basis elements $e_q$, $e_s\in B$, we have
\begin{eqnarray*}
\Big(e_q\otimes e_s, \sum_{p\in \Pi} b\diamond e_p\otimes
f_p\Big) &=&(e_q,b\diamond e_s)=(-e_q\diamond e_s-e_s\diamond e_q, b).
\end{eqnarray*}
Similarly, we obtain
\begin{eqnarray*}
\Big(e_q\otimes e_s, \sum_{i,j,\beta} b_{1i\beta}\otimes b_{2j\beta}\Big)=(e_q\diamond
e_s, b),~~~\Big(e_q\otimes e_s, \sum_{p\in \Pi} e_p\otimes
f_p\diamond b\Big)=(e_q\diamond e_s,b).
\end{eqnarray*}
 Hence by the nondegeneracy of $(\cdot, \cdot)$, we get
\begin{eqnarray*}
&&\sum_{p\in \Pi} b\diamond e_p\otimes f_p=\sum_{p\in \Pi}b\diamond f_p\otimes e_p=\sum_{p\in \Pi} e_p\otimes b\diamond f_p=-\sum_{i,j,\beta}(b_{1i\beta}\otimes b_{2j\beta}+b_{2j\beta}\otimes b_{1i\beta}),\\
&&\sum_{p\in \Pi}e_p\otimes f_p\diamond b=\sum_{i,j,\beta}
b_{1i\beta}\otimes b_{2j\beta},\qquad\quad\sum_{p\in \Pi}e_p\diamond
b\otimes f_p=\sum_{i,j,\beta} b_{2j\beta}\otimes b_{1i\beta}.
\end{eqnarray*}
Therefore, we obtain
\begin{eqnarray*}
	&&(L_{\circ}(a\otimes b)\hatot  \id+\id \hatot ( L_{\circ}(a\otimes b)-R_{\circ}(a\otimes b)))\sum_{p\in \Pi}\sum_\alpha (x_\alpha\otimes e_p)\otimes (y_\alpha \otimes f_p)\\
	&=&-\sum_{i,j,\beta}\sum_\alpha\Big((a\vartriangleright x_\alpha\otimes b_{1i\beta})\otimes
	(y_\alpha\otimes b_{2j\beta})+(a\vartriangleright x_\alpha\otimes b_{2j\beta})\otimes(y_\alpha\otimes b_{1i\beta})\\
	&&\hspace{0.4cm}+(x_\alpha\vartriangleleft a\otimes b_{2j\beta})\otimes (y_\alpha\otimes b_{1i\beta})
	+(x_\alpha\otimes b_{1i\beta})\otimes (a\vartriangleright y_\alpha\otimes b_{2j\beta})\\
	&&\hspace{0.4cm}+(x_\alpha\otimes b_{2j\beta})\otimes (a\vartriangleright y_\alpha\otimes b_{1i\beta})+(x_\alpha\otimes b_{1i\beta})\otimes (y_\alpha\vartriangleleft a\otimes b_{2j\beta})\Big)\\
	&&\hspace{0.4cm}+(x_\alpha\otimes b_{1i\beta})\otimes (y_\alpha\vartriangleright a\otimes b_{2j\beta})+(x_\alpha\otimes b_{2j\beta})\otimes (a\vartriangleleft y_\alpha\otimes b_{1i\beta})\\
	&&\hspace{0.4cm}+(x_\alpha\otimes b_{1i\beta})\otimes (a\vartriangleleft y_\alpha\otimes b_{2j\beta})\Big)\\
	&=& \delta(a\otimes b).
\end{eqnarray*}

This completes the proof.
\end{proof}

\noindent{\bf Acknowledgments.} This research is supported by the Zhejiang
Provincial Natural Science Foundation of China (No. Z25A010006) and
National Natural Science Foundation of China (No. 12171129).

\smallskip

\noindent
{\bf Declaration of interests. } The authors have no conflicts of interest to disclose.

\smallskip

\noindent
{\bf Data availability. } Data sharing is not applicable to this article as no new data were created or analyzed in this study.

\end{document}